\documentclass[11pt,a4paper]{article}
\usepackage[utf8]{inputenc}
\usepackage{amsmath}
\usepackage{amsfonts}
\usepackage{amssymb}
\usepackage{amsthm}
\usepackage{bbm}
\usepackage{xcolor}
\usepackage{graphicx}
\usepackage{hyperref}
\usepackage[top=2cm,
  bottom=2cm,
  left=2cm,
  right=2cm]{geometry}
\numberwithin{equation}{section}
\numberwithin{equation}{section}
\newtheorem{theorem}{Theorem}[section]
\newtheorem{remark}[theorem]{Remark}
\newtheorem{lemma}[theorem]{Lemma}

\newtheorem{corollary}[theorem]{Corollary}
\newtheorem{notation}[theorem]{Notation}

\theoremstyle{definition}
\newtheorem{example}[theorem]{Example}

\newtheorem{assumption}{Assumptions}

\title{Drift-implicit Euler scheme for sandwiched processes \\ driven by H\"older noises}

\author{ Giulia Di Nunno$^{1,2}$\\\href{mailto:giulian@math.uio.no}{giulian@math.uio.no}
   \and  Yuliya Mishura$^3$\\\href{mailto:myus@univ.kiev.ua}{myus@univ.kiev.ua}
   \and Anton Yurchenko-Tytarenko$^1$ \\ \href{mailto:antony@math.uio.no}{antony@math.uio.no}}

\date{%
    $^1$Department of Mathematics, University of Oslo\\[2ex]%
    $^2$Department of Business and Management Science, NHH Norwegian School of Economics, Bergen\\[2ex]%
    $^3$Department of Probability Theory, Statistics and Actuarial Mathematics, Taras Shevchenko National University of Kyiv\\[2ex]
    \today
}

\begin{document}
\maketitle
\begin{abstract}
    In this paper, we analyze the drift-implicit (or backward) Euler numerical scheme for a class of stochastic differential equations with unbounded drift driven by an arbitrary $\lambda$-H\"older continuous process, $\lambda\in(0,1)$. We prove that, under some mild moment assumptions on the H\"older constant of the noise, the $L^r(\Omega;L^\infty([0,T]))$-rate of convergence is equal to $\lambda$. To exemplify, we consider numerical schemes for the generalized Cox--Ingersoll-Ross and Tsallis--Stariolo--Borland models. The results are illustrated by simulations.   
\end{abstract}

\noindent\textbf{Keywords:} sandwiched process, unbounded drift, H\"older continuous noise, numerical scheme\\
\textbf{MSC 2020:} 60H10; 60H35; 60G22; 91G30

\section*{Introduction}

We analyze the \textit{drift-implicit} (also known as \emph{backward}) Euler numerical scheme for stochastic differential equations (SDEs) of the form
\begin{equation}\label{eq: volatility introduction}
    Y(t) = Y(0) + \int_0^t b(s, Y(s))ds + Z(t), \quad t\in [0,T],
\end{equation}
where $Z$ is a general $\lambda$-H\"older continuous noise, $\lambda\in(0,1)$, and the drift $b$ is unbounded and has one of the following two properties:
\begin{itemize}
    \item[\textbf{(A)}] $b(t,y)$ has an explosive growth of the type $(y - \varphi(t))^{-\gamma}$ as $y \downarrow \varphi(t)$, where $\varphi$ is a given H\"older continuous function of the same order $\lambda$ as $Z$ and $\gamma > \frac{1}{\lambda} - 1$;
    \item[\textbf{(B)}] $b(t,y)$ has an explosive growth of the type $(y - \varphi(t))^{-\gamma}$ as $y \downarrow \varphi(t)$ \emph{and} an explosive decrease of the type $-(\psi(t) - y)^{-\gamma}$ as $y \uparrow \psi(t)$, where $\varphi$ and $\psi$ are given H\"older continuous functions  of the same order $\lambda$ as $Z$ such that $\varphi(t) < \psi(t)$, $t\in[0,T]$, and $\gamma > \frac{1}{\lambda} - 1$.
\end{itemize}

The SDEs of this type were extensively studied in \cite{DNMYT2020}. It was shown that the properties \textbf{(A)} or \textbf{(B)}, along with some relatively weak additional assumptions, ensure that the solution to \eqref{eq: volatility introduction} is bounded from below (\textit{one-sided sandwich case}) by the function $\varphi$ in the setting \textbf{(A)}, i.e.
\begin{equation}\label{eq: one-sided introduction}
    Y(t) > \varphi(t), \quad t\in[0,T],
\end{equation}
or stays between $\varphi$ and $\psi$ (\textit{two-sided sandwich case}) in the setting \textbf{(B)}, i.e.
\begin{equation}\label{eq: two-sided introduction}
    \varphi(t) < Y(t) < \psi(t), \quad t\in[0,T].
\end{equation}

We emphasize that the SDE type \eqref{eq: volatility introduction} includes and generalizes several widespread stochastic models. For example, the process given by
\[
    Y(t) = Y(0) - \int_0^t \frac{\kappa Y(s)}{1-Y^2(s)}ds + Z(t), \quad t\in[0,T],
\]
where $Z$ is $\lambda$-H\"older continuous with $\lambda > \frac{1}{2}$, fits into the setting \textbf{(B)} and can be regarded as a natural extension of the Tsallis--Stariolo--Borland (TSB) model employed in biophysics (for more details on the standard Brownian TSB model see e.g. \cite[Subsection 2.3]{Domingo2019} or \cite[Chapter 3 and Chapter 8]{BoundedNoises2013}). Another important example is
\begin{equation}\label{eq: CIR-CEV introduction 1}
    Y(t) = Y(0) + \int_0^t \left(\frac{\kappa_1}{Y^{\gamma}(s)} - \kappa_2 Y(s)\right)ds + Z(t), \quad t\in[0,T],
\end{equation}
where $Z$ is $\lambda$-H\"older continuous, $\lambda\in(0,1)$, and $\gamma > \frac{1}{\lambda} - 1$. It can be shown (see \cite[Subsection 4.2]{DNMYT2020}) that, if $\lambda > \frac{1}{2}$, stochastic process $X(t) := Y^{1+\gamma}(t)$ satisfies the SDE
\begin{equation}\label{eq: CIR/CEV introduction 2}
    X(t) = X(0) + (1+\gamma)\int_0^t\left( \kappa_1  - \kappa_2 X(s) \right) ds + \int_0^t X^\alpha(s) dZ(s), \quad t\in[0,T],
\end{equation}
where $\alpha := \frac{\gamma}{1+\gamma} \in (0,1)$ and the integral w.r.t. $Z$ exists as a pathwise limit of Riemann-Stieltjes integral sums. Equations of the type \eqref{eq: CIR/CEV introduction 2} are used in finance in the standard Brownian setting and are called \textit{Chan–-Karolyi–-Longstaff–-San\-ders} (CKLS) or \textit{constant elasticity of variance} (CEV) model (see, e.g., \cite{Andersen2006, CKLS, Cox1975NotesOO}). If $\alpha = \frac{1}{2}$, the equation \eqref{eq: CIR/CEV introduction 2} is also known as the \textit{Cox--Ingersoll--Ross} (CIR) equation, see , e.g., \cite{COX1981, COX1985-1, COX1985-2}.

In this work, we develop a numerical approximation (both pathwise and in $L^r(\Omega; L^\infty ([0,T]))$) for sandwiched processes \eqref{eq: volatility introduction} which is similar to the \textit{drift-implicit} (also known as \textit{backward}) Euler scheme constructed for the classical Cox-Ingersoll-Ross process in \cite{Alfonsi_2005, Alfonsi_2013, Dereich_Neuenkirch_Szpruch_2012} and extended to the case of the fractional Brownian motion with $H> \frac 1 2$ in \cite{Hong2020, Kubilius_Medziunas_2020, ZhYu2020}. In this drift-implicit scheme, in order to generate $\widehat Y(t_{k+1})$, one has to solve the equation of the type
\begin{equation}\label{eq: equation for BES introduction}
    \widehat Y(t_{k+1}) = \widehat Y(t_{k}) + b(t_{k+1}, \widehat Y(t_{k+1}))\Delta_N + (Z(t_{k+1}) - Z(t_k))
\end{equation}
with respect to $\widehat Y(t_{k+1})$ which is in general a more computationally heavy problem in comparison to the standard Euler-type techniques (see e.g. \cite[Section 5]{DNMYT2020}). However, this drift-implicit numerical method also has a substantial advantage: the approximation $\widehat Y$ maintains the property of being sandwiched, i.e., for all points $t_k$ of the partition 
\[
    \widehat Y(t_k) > \varphi(t_k)
\]
in the setting \textbf{(A)} and 
\[
    \varphi(t_k) < \widehat Y(t_k) < \psi(t_k)
\]
in the case \textbf{(B)}. Having this in mind, we shall say that the drift-implicit scheme is \textit{sandwich preserving}.

We note that a similar approximation scheme was studied in \cite{Kubilius_Medziunas_2020} and \cite{Hong2020, ZhYu2020} for processes of the type \eqref{eq: CIR-CEV introduction 1} driven by a fractional Brownian motion with $H>1/2$. Our work can be seen as an extension of those. However, we emphasize that our results have several elements of novelty. In particular, the paper \cite{Kubilius_Medziunas_2020} discusses only pathwise convergence and not convergence in $L^r(\Omega; L^\infty ([0,T]))$. The approach of \cite{Hong2020} and \cite{ZhYu2020} is very noise specific as both use Malliavin calculus techniques in the spirit of \cite[Proposition 3.4]{Hu2008} to estimate inverse moments of the considered process (which turns out to be crucial to control explosive growth of the drift). As a result, two limitations appear: a restrictive condition involving the time horizon $T$ (see e.g. \cite[Eq. (8) and Remark 3.1]{Hong2020}) and sensitivity to the choice of the noise, i.e. their method cannot be applied directly for drivers other then fBm with $H>1/2$. This lack of flexibility in terms of the choice of the noise is a crucial disadvantage in e.g. finance where modern empirical studies justify the use of fBm with extremely low Hurst index ($H<0.1$) \cite{Bayer_2016} or even drivers with time-varying roughness \cite{Alfi_Coccetti_Petri_Pietronero_2007}. Our approach makes use of \cite[Theorem 3.2]{DNMYT2020} based on the pathwise calculus and allows us to obtain strong convergence with no limitations on $T$ for a substantially larger class of noises. In fact, we require only H\"older continuity of the noise and some moment condition on the corresponding H\"older coefficient which is often satisfied and shared by e.g. \emph{all} H\"older continuous Gaussian processes.  

The paper is organized as follows. Section \ref{sec: assumptions} describes the setting in detail and contains some necessary statements on the properties of the sandwiched processes. In Section \ref{sec: two-sided case}, we give the convergence results in the setting \textbf{(B)} which turns out to be a bit simpler then \textbf{(A)} due to boundedness of the process. Section \ref{sec: one-sided case} extends the scheme to the setting \textbf{(A)}. In Section \ref{sec: examples}, we give some examples and simulations; in particular we show that in some cases (e.g. for the generalized TSB and CIR models) equations \eqref{eq: equation for BES introduction} can be solved explicitly which drastically improves the computational efficiency of the algorithm.

\section{Preliminaries and assumptions}\label{sec: assumptions}

Fix $T>0$ and define
\begin{equation}\label{eq: definition of the set D}
\begin{aligned}
    \mathcal D_{a_1} &:= \{(t,y)\in[0,T]\times\mathbb R_+,~y\in(\varphi(t) + a_1, \infty)\}, \quad & a_1 &> 0,
    \\
    \mathcal D_{a_1, a_2} &:= \{(t,y)\in[0,T]\times\mathbb R_+,~y\in(\varphi(t) + a_1, \psi(t) - a_2)\},  \quad & a_1,a_2 &\in \left[0, \frac{1}{2}\lVert \psi - \varphi \rVert_\infty\right).
\end{aligned}
\end{equation}
where $\varphi$, $\psi \in C([0,T])$ are such that $\varphi(t) < \psi(t)$, $t\in[0,T]$.

\noindent Throughout the paper, we will be dealing with a stochastic differential equation of the form 
\begin{equation}\label{eq: sandwiched SDE, general form}
    Y(t) = Y(0) + \int_0^t b(s, Y(s))ds + Z(t), \quad t\in [0,T].
\end{equation}
The noise $Z = \{Z(t),~t\in[0,T]\}$ is always assumed to satisfy the following conditions:
\begin{itemize}
    \item[\textbf{(Z1)}] $Z(0) = 0$ a.s.;
    \item[\textbf{(Z2)}] $Z$ has a.s. $\lambda$-H\"older continuous paths, $\lambda \in (0,1)$, i.e. there exists a positive random variable $\Lambda$ such that
        \[
            |Z(t) - Z(s)| \le \Lambda |t-s|^\lambda, \quad s,t\in[0,T].
        \]
\end{itemize}

Given the noise $Z$ satisfying \textbf{(Z1)}--\textbf{(Z2)}, the initial value $Y(0)$ and the drift $b$ satisfy \textit{one of the two} lists of assumptions given below.

\begin{assumption}\label{assum: A}{(\textit{One-sided sandwich case})} 
    There exists a $\lambda$-H\"older continuous function $\varphi$: $[0,T] \to \mathbb R$ with $\lambda$ being the same as in \textbf{(Z2)} such that
    \begin{itemize}
        \item[\textbf{(A1)}] $Y(0)$ is deterministic and $Y(0) > \varphi(0)$,
        \item[\textbf{(A2)}] $b$: $\mathcal D_{0} \to \mathbb R$ is continuous and for any $\varepsilon \in \left(0, 1\right)$
        \[
            |b(t_1,y_1) - b(t_2, y_2)| \le \frac{c_1}{\varepsilon^{p}} \left(|y_1 - y_2| + |t_1 - t_2|^\lambda \right), \quad (t_1, y_1), (t_2,y_2) \in \mathcal D_{\varepsilon},
        \]
        where $c_1 > 0$ and $p>1$ are some given constants and $\lambda$ is from \textbf{(Z2)},
        \item[\textbf{(A3)}] 
        \[
            b(t, y) \ge \frac{c_2}{(y - \varphi(t))^{\gamma}}, \quad (t,y) \in \mathcal D_{0}\setminus \mathcal D_{y_*},
        \]
        where $y_*$, $c_2 > 0$ are some given constants and $\gamma > \frac{1}{\lambda} - 1$ with $\lambda$ being from \textbf{(Z2)},
        \item[\textbf{(A4)}] the partial derivative $\frac{\partial b}{\partial y}$ with respect to the spacial variable exists, is continuous and bounded from above, i.e.
        \[
            \frac{\partial b}{\partial y}(t, y) < c_3, \quad (t,y) \in \mathcal D_{0},
        \]
        for some $c_3 > 0$.
    \end{itemize}
\end{assumption}

\begin{assumption}\label{assum: B}{(\textit{Two-sided sandwich case})} There exist $\lambda$-H\"older continuous functions $\varphi$, $\psi$: $[0,T]\to\mathbb R$, $\varphi(t) < \psi(t)$, $t\in[0,T]$, with $\lambda$ being the same as in \textbf{(Z2)} such that
    \begin{itemize}
        \item[\textbf{(B1)}] $Y(0)$ is deterministic and $\varphi(0) < Y(0) < \psi(0)$,
        \item[\textbf{(B2)}] $b$: $\mathcal D_{0,0} \to \mathbb R$ is continuous and for any $\varepsilon \in \left(0, \min\left\{1, \frac{1}{2}\lVert \psi - \varphi \rVert_\infty\right\}\right)$
        \[
            |b(t_1,y_1) - b(t_2, y_2)| \le \frac{c_1}{\varepsilon^{p}} \left(|y_1 - y_2| + |t_1 - t_2|^\lambda \right), \quad (t_1, y_1), (t_2,y_2) \in \mathcal D_{\varepsilon, \varepsilon},
        \]
        where $c_1 > 0$ and $p>1$ are some given constants and $\lambda$ is from \textbf{(Z2)},
        \item[\textbf{(B3)}] 
        \[
            b(t, y) \ge \frac{c_2}{(y - \varphi(t))^{\gamma}}, \quad (t,y) \in \mathcal D_{0,0}\setminus \mathcal D_{y_*, 0},
        \]
        \[
            b(t, y) \le -\frac{c_2}{(\psi(t) - y)^{\gamma}}, \quad (t,y) \in \mathcal D_{0,0}\setminus \mathcal D_{0, y_*},
        \]
        where $y_*$, $c_2 > 0$ are some given constants and $\gamma > \frac{1}{\lambda} - 1$ with $\lambda$ being from \textbf{(Z2)},
        \item[\textbf{(B4)}] the partial derivative $\frac{\partial b}{\partial y}$ with respect to the spacial variable exists, is continuous and bounded from above, i.e.
        \[
            \frac{\partial b}{\partial y}(t, y) < c_3, \quad (t,y) \in \mathcal D_{0,0},
        \]
        for some $c_3 > 0$.
    \end{itemize}
\end{assumption}

Both Assumptions \ref{assum: A} and \ref{assum: B} along with \textbf{(Z1)}--\textbf{(Z2)} ensure that the SDE \eqref{eq: sandwiched SDE, general form} has a unique solution. In the theorem below, we provide some relevant results related to sandwiched processes (see \cite[Theorems 2.3, 2.5, 2.6, 3.1 and 3.2]{DNMYT2020}).

\begin{theorem}\label{th: properties of sandwich} 
    Let $Z=\{Z(t),~t\in[0,T]\}$ be a stochastic process satisfying \textbf{(Z1)}--\textbf{(Z2)}.
    \begin{itemize}
        \item[1)] If the initial value $Y(0)$ and the drift $b$ satisfy assumptions \textbf{(A1)}--\textbf{(A3)}, then the SDE has a unique strong pathwise solution such that for all $t \in [0,T]$
        \begin{equation}\label{eq: one sandwich 1}
            Y(t) > \varphi(t) \quad a.s.
        \end{equation}
        Moreover, there exist deterministic constants $L_1$, $L_2$, $L_3$ and $L_4 > 0$ depending only on $Y(0)$, the shape of $b$ and $\lambda$, such that for all $t\in [0,T]$ the estimate \eqref{eq: one sandwich 1} can be refined as follows:
        \begin{equation}\label{eq: upper and lower bounds for one sandwiched volatility}
            \varphi(t) + \frac{L_1}{ ( L_2 + \Lambda )^{\frac{1}{\gamma \lambda + \lambda - 1}} } \le Y(t) \le L_3 + L_4 \Lambda \quad a.s.,
        \end{equation}
        where $\Lambda$ is from \textbf{(Z2)} and $\gamma$ is from \textbf{(A3)}. In particular, if $\Lambda$ is such that 
        \begin{equation}\label{eq: moment condition on Lambda 1}
            \mathbb E \left[\Lambda^{\frac{r}{\gamma\lambda+\lambda-1}}\right] < \infty
        \end{equation}
        for some $r>0$, then
        \[
            \mathbb E\left[ \sup_{t\in[0,T]} \frac{1}{(Y(t) - \varphi(t))^r} \right] < \infty,
        \]
        and, if 
        \begin{equation}\label{eq: moment condition on Lambda 2}
            \mathbb E \Lambda^{r} < \infty
        \end{equation}
        for some $r>0$, then
        \[
            \mathbb E \left[\sup_{t\in[0,T]} |Y(t)|^r\right] < \infty.
        \]
        
        %%%%%%%%%%%%%%%%%%%%%%%%%%%%%%%%%%%%%%%%%%%%%%%%%%%%%
        
        \item[2)] If the initial value $Y(0)$ and the drift $b$ satisfy assumptions \textbf{(B1)}--\textbf{(B3)}, then the SDE has a unique strong pathwise solution such that for all $t \in [0,T]$
        \begin{equation}\label{eq: double sandwich 1}
            \varphi(t) < Y(t) < \psi(t) \quad a.s.
        \end{equation}
        Moreover, there exist deterministic constants $L_1$ and $L_2 > 0$ depending only on $Y(0)$, the shape of $b$ and $\lambda$, such that for all $t\in [0,T]$ the estimate \eqref{eq: double sandwich 1} can be refined as follows:
        \begin{equation}\label{eq: upper and lower bounds for sandwiched volatility, general case}
            \varphi(t) + \frac{L_1}{ ( L_2 + \Lambda )^{\frac{1}{\gamma \lambda + \lambda - 1}} } \le Y(t) \le \psi(t) - \frac{L_1}{ ( L_2 + \Lambda )^{\frac{1}{\gamma \lambda + \lambda - 1}} } \quad a.s.,
        \end{equation}
        where $\Lambda$ is from \textbf{(Z2)} and $\gamma$ is from \textbf{(B3)}. In particular, if $\Lambda$ can be chosen in such a way that
        \begin{equation}\label{eq: moment condition on Lambda 3}
            \mathbb E \left[ \Lambda^{\frac{r}{\gamma\lambda+\lambda-1}} \right] < \infty
        \end{equation}
        for some $r>0$, then
        \[
            \mathbb E\left[ \sup_{t\in[0,T]} \frac{1}{(Y(t) - \varphi(t))^r} \right] < \infty, \quad \mathbb E\left[ \sup_{t\in[0,T]} \frac{1}{(\psi(t) - Y(t))^r} \right] < \infty.
        \]
    \end{itemize}
\end{theorem} 

\begin{remark}\label{rem: pathwise}
    Properties \eqref{eq: one sandwich 1}--\eqref{eq: upper and lower bounds for one sandwiched volatility} and \eqref{eq: double sandwich 1}--\eqref{eq: upper and lower bounds for sandwiched volatility, general case} hold on each $\omega \in \Omega$ such that $Z(\omega; t)$, $t\in[0,T]$, is H\"older continuous and we always consider only such $\omega\in\Omega$ in all proofs with pathwise arguments. For notational simplicity, we will also omit $\omega$ in brackets. 
\end{remark}

\begin{remark}
    Due to the property \eqref{eq: double sandwich 1}, the setting described in Assumptions \ref{assum: B} will be referred to as the \textbf{two-sided sandwich case} since the solution is ``sandwiched'' between $\varphi$ and $\psi$ a.s. Similarly, the property \eqref{eq: one sandwich 1} justifies the name \textbf{one-sided sandwich case} for the setting corresponding to Assumptions \ref{assum: A}. In both cases \ref{assum: A} and \ref{assum: B}, the solution to \eqref{eq: sandwiched SDE, general form} will be referred to as a \textbf{sandwiched process}.
\end{remark}

\begin{remark}
    Note that assumptions \textbf{(A4)} and \textbf{(B4)} are not required for Theorem \ref{th: properties of sandwich} to hold and will be used later on. 
\end{remark}

In what follows, conditions \eqref{eq: moment condition on Lambda 1}, \eqref{eq: moment condition on Lambda 2} and \eqref{eq: moment condition on Lambda 3} will play an important role since the $L^r(\Omega;L^\infty([0,T]))$-convergence of the approximation scheme will directly follow from the integrability of $\Lambda$. However it should be noted that these conditions are not very restricting as indicated in the following example.

\begin{example}{\textit{(H\"older Gaussian noises)}}
    Let $Z=\{Z(t),~t\in[0,T]\}$ be an arbitrary H\"older continuous Gaussian process satisfying \textbf{(Z1)}--\textbf{(Z2)}. In this case, by \cite{ASVY2014}, the random variable $\Lambda$ from \textbf{(Z2)} can be chosen to have moments of \textit{all} orders.
\end{example}

We now complete the Section with some examples of the sandwiched processes.

\begin{example}{(\emph{Generalized CIR and CKLS/CEV models})}
    Let $\varphi \equiv 0$, $Z$ satisfy \textbf{(Z1)}--\textbf{(Z2)} with $\lambda \in(0,1)$ and $Y(0)$, $\kappa_1$, $\kappa_2 >0$, $\gamma > \frac{1}{\lambda} - 1$ be given. Then, by Theorem \ref{th: properties of sandwich}, 1), the SDE of the form
    \begin{equation}\label{eq: CIR/CEV with additive noise}
        Y(t) = Y(0) + \int_0^t \left(\frac{\kappa_1}{Y^{\gamma}(s)} - \kappa_2 Y(s)\right)ds + Z(t), \quad t\in[0,T],
    \end{equation}
    has a unique positive solution. Moreover, it can be shown (see \cite[Subsection 4.2]{DNMYT2020}) that, if $\lambda > \frac{1}{2}$, stochastic process $X(t) := Y^{1+\gamma}(t)$, $t\in[0,T]$, a.s. satisfies the SDE of the form
    \begin{equation}\label{eq: CIR/CEV}
        X(t) = X(0) + (1+\gamma)\int_0^t\left( \kappa_1  - \kappa_2 X(s) \right) ds + \int_0^t X^\alpha(s) dZ(s), \quad t\in[0,T],
    \end{equation}
    where $\alpha := \frac{\gamma}{1+\gamma} \in (0,1)$ and the integral w.r.t. $Z$ exists a.s. as a pathwise limit of Riemann-Stieltjes integral sums. As mentioned already, the \eqref{eq: CIR/CEV} appears in finance in the standard Brownian setting and is called \textit{Chan–Karolyi–Longstaff–Sanders} (CKLS) or \textit{constant elasticity of variance} (CEV) model (see e.g. \cite{Andersen2006, CKLS, Cox1975NotesOO}). If $\alpha = \frac{1}{2}$ (i.e. when $\gamma = 1$), the equation \eqref{eq: CIR/CEV} is also known as the \textit{Cox-Ingersoll-Ross} (CIR) equation \cite{COX1981, COX1985-1, COX1985-2}.
\end{example}

\begin{example}\label{ex: TSB}{(\textit{Generalized TSB model})}
    Let $\varphi \equiv -1$, $\psi \equiv 1$, $Y(0) \in (-1,1)$, $Z$ satisfy \textbf{(Z1)}--\textbf{(Z2)} with $\lambda > \frac{1}{2}$ and $\kappa > 0$. Then, by Theorem \ref{th: properties of sandwich}, 2), the SDE of the form
    \begin{equation}\label{eq: TSB model}
        Y(t) = Y(0) - \int_0^t \frac{\kappa Y(s)}{1-Y^2(s)}ds + Z(t), \quad t\in[0,T],
    \end{equation}
    has a unique solution such that $-1 < Y(t) < 1$ for all $t\in[0,T]$ a.s. In the standard Brownian setting, the SDE of the type \eqref{eq: TSB model} is known as the Tsallis--Stariolo--Borland (TSB) model and is used in biophysics (for more details, see e.g. \cite[Subsection 2.3]{Domingo2019} or \cite[Chapter 3 and Chapter 8]{BoundedNoises2013}).
\end{example}

\begin{example}\label{ex: generalized TSB}
    For the given $Z$ satisfying \textbf{(Z1)}--\textbf{(Z2)} with $\lambda\in(0,1)$,  $\lambda$-H\"older continuous functions $\varphi$, $\psi$, $\varphi(t) < \psi(t)$, $t\in[0,T]$, and $Y(0) \in (\varphi(0), \psi(0))$ consider the SDE of the form
    \[
        Y(t) = Y(0) + \int_0^t \left(\frac{\kappa_1}{(Y(s) - \varphi(s))^\gamma}-  \frac{\kappa_2}{(\psi(s) - Y(s))^\gamma} - \kappa_3 Y(s)\right)ds + Z(t), \quad t\in[0,T],
    \]
    where $\kappa_1$, $\kappa_2 > 0$, $\kappa_3\in \mathbb R$ and $\gamma > \frac{1}{\lambda} - 1$. By Theorem \ref{th: properties of sandwich}, 2), this SDE has a unique solution such that $\varphi(t) < Y(t) < \psi(t)$ a.s. Note that the TSB drift from \eqref{eq: TSB model} also has this shape with $\varphi \equiv -1$, $\psi \equiv 1$, $\gamma = 1$, $\kappa_1 = \kappa_2 = \frac{\kappa}{2}$ and $\kappa_3 = 0$ since
    \[
        -\frac{\kappa y}{1 - y^2} = \frac{\kappa}{2}\left( \frac{1}{y+1} - \frac{1}{1-y} \right).
    \]
\end{example}

\begin{notation}
    In what follows, $C$ denotes any positive deterministic constant that does not depend on the partition and the exact value of which is not relevant. Note that $C$ may change from line to line (or even within one line). 
\end{notation}

\section{The approximation scheme for the two-sided sandwich}\label{sec: two-sided case}

We will start by considering the numerical scheme for the \textit{two-sided sandwich case} which turns out to be slightly simpler due to boundedness of $Y$. Let the noise $Z$ satisfy \textbf{(Z1)}--\textbf{(Z2)}, $Y(0)$ and $b$ satisfy Assumptions \ref{assum: B} and $Y = \{Y(t),~t \in [0,T]\}$ be the unique solution of the SDE \eqref{eq: sandwiched SDE, general form}. Consider a uniform partition $\{0=t_0 < t_1<...<t_N=T\}$ of $[0,T]$, $t_k := \frac{Tk}{N}$, $k=0,1,..., N$, with the mesh $\Delta_N:=\frac{T}{N}$ such that
\begin{equation}\label{eq: condition on the mesh}
    c_3 \Delta_N < 1,
\end{equation}
where $c_3$ is an upper bound for $\frac{\partial b}{\partial y}$ from \textbf{(B4)}. Let us define $\widehat Y(t)$ as follows:
\begin{equation}\label{eq: definition of backward Euler scheme}
\begin{aligned}
    \widehat Y(0) &= Y(0),
    \\
    \widehat Y(t_{k+1}) &= \widehat Y(t_{k}) + b(t_{k+1}, \widehat Y(t_{k+1}))\Delta_N + (Z(t_{k+1}) - Z(t_k)),
    \\
    \widehat Y(t) &= \widehat Y(t_k), \quad t\in[t_k, t_{k+1}),
\end{aligned}
\end{equation}
where the second expression is considered as an equation with respect to $\widehat Y(t_{k+1})$. 

\begin{remark}
    Equation with respect to $\widehat Y(t_{k+1})$ from \eqref{eq: definition of backward Euler scheme} has a unique solution such that $\widehat Y(t_{k+1}) \in (\varphi(t_{k+1}), \psi(t_{k+1}))$. Indeed, for any fixed $t\in[0,T]$ and any $z\in\mathbb R$, consider the equation
    \begin{equation}\label{eq: equation w.r.t. y}
        y - b(t, y)\Delta_N = z
    \end{equation}
    w.r.t. $y$. Assumption \textbf{(B4)} together with condition \eqref{eq: condition on the mesh} imply that $(y - b(t, y)\Delta_N)'_y > 0$ and, by \textbf{(B3)},
    \begin{gather*}
        y - b(t, y)\Delta_N \to -\infty, \quad y \to \varphi(t)+,
        \\
        y - b(t, y)\Delta_N \to \infty, \quad y \to \psi(t)-.
    \end{gather*}
    Thus there exists a unique $y \in (\varphi(t), \psi(t))$ satisfying \eqref{eq: equation w.r.t. y}.
\end{remark}
\begin{remark}
    The value of $\widehat Y(t)$ for $t\in[0,T]\setminus\{t_0,...,t_N\}$ can also be defined via linear interpolation as
    \[
        \widehat Y(t) = \frac{1}{\Delta_N}\left( (t_{k+1} - t) \widehat Y(t_k) + (t - t_k) \widehat Y(t_{k+1}) \right), \quad t\in[t_k, t_{k+1}), \quad k=0,...,N-1.
    \]
    In such case all results of this section hold with almost no changes in the proofs.
\end{remark}

\begin{remark}
    The algorithms of the type \eqref{eq: definition of backward Euler scheme} are sometimes called the \textbf{drift-implicit} \cite{Alfonsi_2005, Alfonsi_2013, Dereich_Neuenkirch_Szpruch_2012} or \textbf{backward} \cite{Hong2020} Euler approximation schemes.
\end{remark}

Before presenting the main results of this section, we require some auxiliary lemmas. First of all, we note that the values $\widehat Y(t_n)$, $n=0,1,...,N$, of the discretized process are bounded away from both $\varphi$ and $\psi$ by random variables that do not depend on the partition. Namely, we have the following result that can be regarded as a discrete modification of arguments in \cite[Theorem 3.2]{DNMYT2020}.

\begin{lemma}\label{lemma: bounds for discretized process}
    Let $Z$ satisfy \textbf{(Z1)}--\textbf{(Z2)}, Assumptions \ref{assum: B} hold and the mesh of the partition $\Delta_N$ satisfy \eqref{eq: condition on the mesh}. Then there exist deterministic constants $L_1$ and $L_2>0$ depending only on $Y(0)$, the shape of the drift $b$ and $\lambda$, such that
    \[
        \varphi(t_n) + \frac{L_1}{ ( L_2 + \Lambda )^{ \frac{1}{\gamma\lambda + \lambda - 1} } } \le \widehat Y(t_n) \le \psi(t_n) - \frac{L_1}{ ( L_2 + \Lambda )^{\frac{1}{\gamma\lambda + \lambda - 1}} }, \quad n = 0, 1,... ,N, \quad a.s.,
    \]
    where $\Lambda$ is from \textbf{(Z2)} and $\gamma$ is from \textbf{(B3)}. 
\end{lemma}

\begin{proof}
    We will prove that
    \begin{equation}\label{eq: lower bound for discretised process, proof}
        \varphi(t_n) + \frac{L_1}{ ( L_2 + \Lambda )^{\alpha} } \le \widehat Y(t_n), \quad n = 0, 1,... ,N, \quad a.s.
    \end{equation}
    by using the pathwise argument (see Remark \ref{rem: pathwise}). The other inequality can be derived in a similar manner. Recall that, by Assumptions \ref{assum: B}, $\varphi$ and $\psi$ are $\lambda$-H\"older continuous, i.e. there exists $K > 0$ such that
    \[
        |\varphi(t) - \varphi(s)| + |\psi(t) - \psi(s)| \le K |t-s|^{\lambda}, \quad t,s\in[0,T].
    \]
    Denote also
    \[
        \beta := \frac{    \lambda^{\frac{\lambda}{1 - \lambda}} -  \lambda^{\frac{1}{1 - \lambda}} }{ c_2^{\frac{\lambda}{1 - \lambda}}} > 0,
    \]
    where $c_2$ is from \textbf{(B3)},
    \[
        L_2 := K + \left(2 \beta\right)^{\lambda - 1} \left(\frac{(Y(0) - \varphi(0)) \wedge y_* \wedge (\psi(0) - Y(0))}{2}\right)^{1 - \lambda - \gamma\lambda} > 0,
    \]
    with the constants $y_*$ and $\gamma$ also from \textbf{(B3)}, and
    \[
        \varepsilon := \frac{1}{(2 \beta )^{\frac{1 - \lambda}{\gamma \lambda + \lambda -1}} (L_2 + \Lambda) ^{\frac{1 }{\gamma \lambda + \lambda -1}}}.
    \]
    Note that, with probability 1, 
    \[
        |\varphi(t) - \varphi(s)| + |\psi(t) - \psi(s)| + |Z(t) - Z(s)| \le (L_2 + \Lambda) |t-s|^{\lambda}, \quad t, s \in[0,T],
    \]
    and, furthermore, it is easy to check that $\varepsilon < Y(0) - \varphi(0)$, $\varepsilon < \psi(0) - Y(0)$ and $\varepsilon < y_*$.
    
    If $\widehat Y(t_n) \ge \varphi(t_n) + \varepsilon$ for a particular $n=0,1,...,N$, then, by definition of $\varepsilon$, the bound of the type \eqref{eq: lower bound for discretised process, proof} holds automatically. Suppose that there exists $n = 1,...,N$ such that $\widehat Y({t_n}) < \varphi(t_n) + \varepsilon$. Denote by $\kappa(n)$ the last point of the partition before $t_n$ on which $\widehat Y$ stays above $\varepsilon$, i.e. 
    \[
        \kappa(n) := \max\{k = 0,..., n-1 ~|~\widehat Y_{t_{k}} \ge \varphi(t_k) + \varepsilon\}
    \]
    (note that such point exists since $\widehat Y(t_0) - \varphi(0) = Y(0) - \varphi(0) > \varepsilon$). Then, for all $k=\kappa(n)+1,..., n$ we have that $\widehat Y(t_k) < \varepsilon < y_*$ and therefore, using \textbf{(B3)}, we obtain that, with probability 1,
    \begin{align*}
        \widehat Y(t_n)  - \varphi(t_n) &= \widehat Y(t_{\kappa(n)}) - \varphi(t_n) + \Delta_N \sum_{k = \kappa(n) + 1}^n b(t_k, \widehat Y(t_k) ) + Z(t_n) - Z(t_{\kappa(n)}) 
    \\
        &\ge \varepsilon + \varphi(t_{\kappa(n)}) - \varphi(t_n) + \frac{c_2}{\varepsilon^\gamma}(t_n - t_{\kappa(n)}) + Z(t_n) - Z(t_{\kappa(n)})
    \\
        &\ge \varepsilon + \frac{c_2}{\varepsilon^\gamma}(t_n - t_{\kappa(n)}) - (L_2+\Lambda) (t_n - t_{\kappa(n)})^\lambda.
    \end{align*}
    Consider a function $F_{\varepsilon} : \mathbb R_+ \to \mathbb R$ such that
    \[
        F_{\varepsilon} (t) = \varepsilon + \frac{ c_2 }{\varepsilon^{\gamma}} t - (L_2+\Lambda) t^{\lambda}.
    \]
    It is straightforward to verify that $F_\varepsilon$ attains its minimum at
    \[
        t_* := \left(\frac{\lambda}{c_2}\right)^{\frac{1}{1 - \lambda}} \varepsilon^{\frac{\gamma}{1 - \lambda}} (L_2 + \Lambda)^{\frac{1}{1 - \lambda}}
    \]
    and, taking into account the explicit form of $\varepsilon$,
    \begin{align*}
        F_\varepsilon(t_*) &= \varepsilon + \frac{\lambda^{\frac{1}{1 - \lambda}}}{ c_2^{\frac{\lambda}{1 - \lambda}}} \varepsilon^{\frac{\gamma \lambda}{1 - \lambda}} (L_2 + \Lambda)^{\frac{1}{1 - \lambda}} 
        - \frac{\lambda^{\frac{\lambda}{1 - \lambda}}}{c_2^{\frac{\lambda}{1 - \lambda}}} \varepsilon^{\frac{\gamma \lambda}{1 - \lambda}} (L_2 + \Lambda)^{\frac{1}{1 - \lambda}}
    \\
        &= \varepsilon - \beta \varepsilon^{\frac{\gamma \lambda}{1 - \lambda}} (L_2 + \Lambda)^{\frac{1}{1 - \lambda}} = \frac{1}{2^{\frac{\gamma\lambda }{\gamma \lambda + \lambda -1}} \beta^{\frac{1 - \lambda}{\gamma \lambda + \lambda -1}} (L_2 + \Lambda) ^{\frac{1 }{\gamma \lambda + \lambda -1}}} 
        \\
        & = \frac{\varepsilon}{2}.
    \end{align*}
    Namely, even if $\widehat Y(t_n) < \varphi(t_n) + \varepsilon$, we still have that, with probability 1,
    \[
        \widehat Y(t_n)  - \varphi(t_n) \ge F_{\varepsilon} ( t_n - t_{\kappa(n)} ) \ge F_\varepsilon(t_*) = \frac{\varepsilon}{2},
    \]
    and thus, with probability 1, for any $n=0,1,...,N$
    \begin{align*}
        \widehat Y(t_n)  &\ge  \varphi(t_n) + \frac{\varepsilon}{2} = \varphi(t_n) + \frac{1}{2^{\frac{\gamma\lambda }{\gamma \lambda + \lambda -1}} \beta^{\frac{1 - \lambda}{\gamma \lambda + \lambda -1}} (L_2 + \Lambda) ^{\frac{1 }{\gamma \lambda + \lambda -1}}} 
        \\
        & =: \varphi(t_n) + \frac{L_1}{(L_2+ \Lambda)^{\frac{1 }{\gamma \lambda + \lambda -1}}},
    \end{align*}
    where $L_1 := \frac{1}{2^{\frac{\gamma\lambda }{\gamma \lambda + \lambda -1}} \beta^{\frac{1 - \lambda}{\gamma \lambda + \lambda -1}} }$.
\end{proof}

\begin{remark}\label{rem: L1, L2 and alpha can be chosen jointly}
    It is clear that constants $L_1$ and $L_2$ in Lemma \ref{lemma: bounds for discretized process} can be chosen jointly for $Y$ and $\widehat Y$, so that the inequalities
    \[
        \varphi(t) + \frac{L_1}{ ( L_2 + \Lambda )^{\alpha} } \le Y(t) \le \psi(t) - \frac{L_1}{ ( L_2 + \Lambda )^{\alpha} }, \quad t\in [0,T],
    \]
    and
    \[
        \varphi(t_n) + \frac{L_1}{ ( L_2 + \Lambda )^{\alpha} } \le \widehat Y(t_n) \le \psi(t_n) - \frac{L_1}{ ( L_2 + \Lambda )^{\alpha} }, \quad n = 0, 1,... ,N,
    \]
    hold simultaneously with probability 1. 
\end{remark}

Next, we proceed with a simple property of the sandwiched process $Y$ in \eqref{eq: sandwiched SDE, general form}.

\begin{lemma}\label{lemma: Holder continuity of Y}
    Let $Z$ satisfy \textbf{(Z1)}--\textbf{(Z2)} and assumptions \textbf{(B1)}--\textbf{(B3)} hold. 
    \begin{itemize}
        \item[1)] There exists a positive random variable $\Upsilon$ such that, with probability 1,
        \[
            |Y(t) - Y(s)| \le \Upsilon |t-s|^\lambda, \quad t, s\in[0,T].
        \]
        \item[2)] If, for some $ r \ge 1$,
        \begin{equation}\label{eq: condition on Lambda to put expectation}
            \mathbb E \left[\Lambda^{\frac{r\max\{p, \gamma\lambda + \lambda - 1\}}{\gamma\lambda + \lambda - 1}}\right] < \infty,
        \end{equation}
        where $\lambda$ and $\Lambda$ are from \textbf{(Z2)}, $p$ is from \textbf{(B2)} and $\gamma$ is from \textbf{(B3)}, then one can choose $\Upsilon$ such that
        \[
            \mathbb E [\Upsilon^{r}] < \infty.
        \]
    \end{itemize}
\end{lemma}

\begin{proof}
    Denote $\phi(t) := \frac{1}{2}(\psi(t) + \varphi(t))$, $t\in[0,T]$. By \eqref{eq: upper and lower bounds for sandwiched volatility, general case},
    \[
        \varphi(t) + \frac{L_1}{ ( L_2 + \Lambda )^{ \frac{1}{\gamma\lambda + \lambda - 1} } } \le Y(t) \le \psi(t) - \frac{L_1}{ ( L_2 + \Lambda )^{\frac{1}{\gamma\lambda + \lambda - 1}} }, \quad t\in [0,T], \quad a.s.,
    \]
    i.e. with probability 1 $(t, Y(t)) \in \mathcal D_{\frac{1}{\xi}, \frac{1}{\xi}}$, $t\in[0,T]$, where
    \begin{equation}\label{eq: distance to pieces of bread denoted by xi}
        \xi := \frac{( L_2 + \Lambda )^{\frac{1}{\gamma\lambda + \lambda - 1}}}{L_1}
    \end{equation}
    and $\mathcal D_{\frac{1}{\xi}, \frac{1}{\xi}}$ is defined by \eqref{eq: definition of the set D}. It is evident that $(t, \phi(t)) \in \mathcal D_{\frac{1}{\xi}, \frac{1}{\xi}}$, $t\in[0,T]$, therefore, using \textbf{(Z2)}, \textbf{(B2)} and \eqref{eq: double sandwich 1}, we can write that, with probability 1, for all $0\le s < t \le T$:
    \begin{equation}\label{proofeq: estimate for Holder cont two sided}
    \begin{aligned}
        |Y(t) - Y(s)| & \le \int_s^t |b(u, Y(u))|du + |Z(t) - Z(s)|
        \\
        & \le \int_s^t |b(u, Y(u)) - b(u, \phi(u))|du + \int_s^t |b(u, \phi(u))| du + \Lambda (t - s)^\lambda
        \\
        & \le c_1 \xi^p \int_s^t |Y(u) - \phi(u)| du + \max_{u\in[0,T]}|b(u, \phi(u))| (t-s) + \Lambda (t - s)^\lambda
        \\
        & \le \left( c_1 \xi^p \lVert \psi - \varphi \rVert_\infty + \max_{u\in[0,T]}|b(u, \phi(u))|\right)(t-s) + \Lambda (t - s)^\lambda
        \\
        & \le C(\xi^p + \Lambda + 1)(t-s)^\lambda,
    \end{aligned}
    \end{equation}
    where $C$ is a positive constant. Now one can put
    \begin{equation}\label{eq: definition of Upsilon}
        \Upsilon := C(\xi^p + \Lambda + 1)
    \end{equation}
    and observe that the definition of $\Upsilon$, \eqref{eq: condition on Lambda to put expectation} and \eqref{eq: distance to pieces of bread denoted by xi} imply that
    \[
        \mathbb E [\Upsilon^{r}] < \infty.
    \]
\end{proof}

\begin{corollary}\label{rem: Holder approximations}
    Under \textbf{(Z1)}--\textbf{(Z2)} and Assumptions \ref{assum: B}, using Lemma \ref{lemma: bounds for discretized process} and following the proof of Lemma \ref{lemma: Holder continuity of Y}, it is easy to obtain that there exists a random variable $\Upsilon$ independent of the partition such that with probability 1
    \begin{equation}\label{eq: pseudo Holder continuity of discretised process}
        |\widehat Y(t_k) - \widehat Y(t_n)| \le \Upsilon |t_k - t_n|^\lambda, \quad k,n = 0,...,N.
    \end{equation}
    Furthermore, just like in Lemma \ref{lemma: Holder continuity of Y}, if \eqref{eq: condition on Lambda to put expectation} for some for $r\ge 1$, then
    \[
        \mathbb E [\Upsilon^r] < \infty.
    \]
    Finally, just as in Remark \ref{rem: L1, L2 and alpha can be chosen jointly}, $\Upsilon$ can be chosen jointly for $Y$ and $\widehat Y$, so that
    \[
        |Y(t) - Y(s)| \le \Upsilon |t-s|^\lambda, \quad t,s\in[0,T],
    \]
    holds simultaneously with \eqref{eq: pseudo Holder continuity of discretised process} with probability 1.
\end{corollary}

\begin{lemma}\label{lemma: backward approximations in discrete time points}
    Let $Z$ satisfy \textbf{(Z1)}--\textbf{(Z2)}, Assumptions \ref{assum: B} hold and the mesh of the partition $\Delta_N$ satisfy \eqref{eq: condition on the mesh}.
    \begin{itemize}
        \item[1)] For any $r\ge 1$, there exists a positive random variable $\mathcal C_1$ that does not depend on the partition such that
        \[
            \sup_{k=0,1,...,N} |Y(t_k) - \widehat{Y}(t_k)|^r \le \mathcal C_1 \Delta_N^{\lambda r}  \quad a.s. 
        \]
        \item[2)] If, additionally,
        \begin{equation}\label{eq: condition on moments of Lambda}
            \mathbb E\left[ \Lambda^{\frac{r(p+\max\{p, \gamma\lambda + \lambda - 1\})}{\gamma\lambda + \lambda -1}} \right] < \infty,
        \end{equation}
        where $\lambda$ and $\Lambda$ are from \textbf{(Z2)}, $p$ is from \textbf{(B2)} and $\gamma$ is from \textbf{(B3)}, then one can choose $\mathcal C_1$ such that  $\mathbb E [\mathcal C_1] < \infty$, i.e. there exists a deterministic constant $C$ that does not depend on the partition such that
        \[
            \mathbb E\left[\sup_{k=0,1,...,N} |Y(t_k) - \widehat{Y}(t_k)|^r \right] \le C\Delta_N^{\lambda r}.
        \]
    \end{itemize}
\end{lemma}

\begin{proof}
    Fix $\omega\in\Omega$ such that $Z(\omega, t)$, $t\in[0,T]$, is H\"older continuous (for simplicity of notation, we will omit $\omega$ in the brackets). Denote $e_n := Y(t_n) - \widehat Y(t_n)$, $\Delta Z_n := Z(t_n) - Z(t_{n-1})$. Then
    \begin{equation}\label{eq: IE error estimate}
    \begin{aligned}
        e_n &= Y(t_{n-1}) + \int_{t_{n-1}}^{t_{n}} b(s, Y(s)) ds + \Delta Z_{n} 
        \\
        &\quad- \widehat Y(t_{n-1}) - b(t_{n}, \widehat Y(t_n)) \Delta_N - \Delta Z_n
        \\
        &= e_{n-1} +  \left(b(t_n, Y(t_n)) - b(t_{n}, \widehat Y(t_n))\right) \Delta_N
        \\
        &\quad + \int_{t_{n-1}}^{t_{n}} (b(s,Y(s)) - b(t_n, Y(t_n)))ds.
    \end{aligned}
    \end{equation}
    By the mean value theorem,
    \begin{equation*}
    \begin{aligned}
        \left(b(t_n, Y(t_n)) - b(t_{n}, \widehat Y(t_n))\right) \Delta_N &= \frac{\partial b}{\partial y}(t_n, \Theta_n) \Delta_N  (Y(t_n) - \widehat Y(t_n)) 
        \\
        &= \frac{\partial b}{\partial y} (t_n, \Theta_n) \Delta_N e_n
    \end{aligned}
    \end{equation*}
    with $\Theta_n \in (Y(t_n) \wedge \widehat Y(t_n), Y(t_n) \vee \widehat Y(t_n) )$. Using this, we can rewrite \eqref{eq: IE error estimate} as follows:
    \begin{equation}\label{eq: IE error estimate 2}
    \begin{aligned}
        \left(1 - \frac{\partial b}{\partial y}(t_n, \Theta_n) \Delta_N \right) e_n &= e_{n-1} + \int_{t_{n-1}}^{t_{n}} (b(s,Y(s)) - b(t_n, Y(t_n)))  ds,
    \end{aligned}
    \end{equation}
    where
    \[
        1 - \frac{\partial b}{\partial y}(t_n, \Theta_n) \Delta_N > 1 - c_3\Delta_N > 0
    \]
    by \textbf{(B4)} and \eqref{eq: condition on the mesh}. 
    
    Next, denote 
    \[
        \zeta_0 := 1, \quad \zeta_{n} := \prod_{i=1}^{n} \left(1 - \frac{\partial b}{\partial y}(t_i, \Theta_i)\Delta_N\right)
    \]
    and define $\tilde e_n := \zeta_n e_n$. By multiplying both sides of \eqref{eq: IE error estimate 2} by $\zeta_{n-1}$, we obtain that
    \begin{equation}\label{eq: IE error estimate 3}
    \begin{aligned}
        \tilde e_n = \tilde e_{n-1} &+ \zeta_{n-1} \int_{t_{n-1}}^{t_{n}} (b(s,Y(s)) - b(t_n, Y(t_n)))  ds
    \end{aligned}
    \end{equation}
    and, expanding the terms $\tilde e_{i-1}$ in \eqref{eq: IE error estimate 3} one by one, $i=n, n-1,...,1$, and taking into account that $\tilde e_0 = 0$, we obtain that
    \begin{equation*}
    \begin{aligned}
        \tilde e_n &= \sum_{i=1}^n \zeta_{i-1} \int_{t_{i-1}}^{t_{i}} (b(s,Y(s)) - b(t_i, Y(t_i))) ds.
    \end{aligned}
    \end{equation*}
    Therefore
    \begin{equation*}
    \begin{aligned}
        e_n &= \sum_{i=1}^n \frac{\zeta_{i-1}}{\zeta_n} \int_{t_{i-1}}^{t_{i}} (b(s,Y(s)) - b(t_i, Y(t_i))) ds.
    \end{aligned}
    \end{equation*}
    Observe that, by assumption \textbf{(B4)} and \eqref{eq: condition on the mesh}, for any $i,n \in \mathbb N$, $i<n$,
    \begin{align*}
        \frac{\zeta_k}{\zeta_n} &= \prod_{i=k+1}^n \left(1 - \frac{\partial b}{\partial y}(t_i, \Theta_i) \Delta_N\right)^{-1} \le \prod_{i=k+1}^N (1 - c_3 \Delta_N)^{-1}
        \\
        &\le  (1 - c_3 \Delta_N)^{-N} = \left(1 - \frac{c_3T}{N}\right)^{-N} \to e^{c_3 T}, \quad N\to\infty,
    \end{align*}
    whence there exists a constant $C$ that does not depend on $i$, $n$ or $N$ such that
    \[
        \frac{\zeta_k}{\zeta_n} \le C.
    \]
    Using this, one can deduce that
    \begin{align*}
        |e_n|^r &\le C \left| \sum_{i=1}^n \frac{\zeta_{i-1}}{\zeta_n} \int_{t_{i-1}}^{t_{i}} (b(s,Y(s)) - b(t_i, Y(t_i))) ds \right|^r 
        \\
        &\le C\left(\sum_{i=1}^n \int_{t_{i-1}}^{t_{i}}\left| b(s,Y(s)) - b(t_i, Y(t_i))\right|  ds\right)^r.
    \end{align*}
    Note that $(t, Y(t)) \in \mathcal D_{\frac{1}{\xi}, \frac{1}{\xi}}$, where $\xi$ is defined by \eqref{eq: distance to pieces of bread denoted by xi} and $\mathcal D_{\frac{1}{\xi}, \frac{1}{\xi}}$ is defined via \eqref{eq: definition of the set D}, hence, by \textbf{(B2)} as well as Lemma \ref{lemma: Holder continuity of Y}, we can deduce that
    \begin{align*}
    &\left(\sum_{i=1}^n \int_{t_{i-1}}^{t_{i}}\left| b(s,Y(s)) - b(t_i, Y(t_i))\right|  ds\right)^r
    \\
        & \le C\xi^{pr}\left(\sum_{i=1}^n \int_{t_{i-1}}^{t_{i}}\left| s - t_i \right|^\lambda  ds\right)^r + C\xi^{pr} \left(\sum_{i=1}^n \int_{t_{i-1}}^{t_{i}}\left| Y(s) - Y(t_i)\right|  ds\right)^r
    \\
        &\le C\xi^{pr}\left(\sum_{i=1}^n \int_{t_{i-1}}^{t_{i}}\left| s - t_i \right|^\lambda  ds\right)^r + C\xi^{pr}\Upsilon^r \left(\sum_{i=1}^n \int_{t_{i-1}}^{t_{i}}\left| s - t_i\right|^\lambda  ds\right)^r
    \\
        & = C\xi^{pr}(1+\Upsilon^r)\left(\sum_{i=1}^n \frac{1}{(1+\lambda)} \Delta_N^{1+\lambda} \right)^r
    \\
        & \le C\xi^{pr}(1+ \Upsilon^r) \Delta_N^{\lambda r}.
\end{align*}
In other words, there exists a constant $C$ that does not depend on the partition such that
\begin{equation*}
    |e_n|^r = |Y(t_n) - \widehat Y(t_n)|^r \le C\xi^{pr}(1+\Upsilon^r) \Delta_N^{\lambda r}
\end{equation*}
and, since the right-hand side of the relation above does not depend on $n$ or $N$, we have
\begin{equation}\label{eq: estimation for distance between Y and hatY}
    \sup_{n = 0,...,N}|Y(t_n) - \widehat Y(t_n)|^r \le C\xi^{pr}(1+\Upsilon^r) \Delta_N^{\lambda r} =: \mathcal C_1 \Delta_N^{\lambda r}.
\end{equation}
It remains to notice that, by \eqref{eq: distance to pieces of bread denoted by xi} and \eqref{eq: definition of Upsilon},
\[
    \mathbb E \left[ \xi^{pr}(1+\Upsilon^r) \right] < \infty
\]
whenever \eqref{eq: condition on moments of Lambda} holds, which finally implies
\[
    \mathbb E \left[\sup_{n=0,...,N} |Y(t_n) - \widehat Y(t_n)|^r \right] \le \mathbb E[\mathcal C_1] \Delta_N^{\lambda r} =: C\Delta_N^{\lambda r}.
\]

\end{proof}

Now we are ready to proceed to the main results of this subsection.
%Recall that $\widehat Y(t) = \widehat Y(t_k)$, $t\in[t_k, t_{k+1})$, see Remark \ref{rem: linear interpolation is not necessary}.

\begin{theorem}\label{th: main result for double sandwich}
    Let $Z$ satisfy \textbf{(Z1)}--\textbf{(Z2)}, Assumptions \ref{assum: B} hold and the mesh of the partition $\Delta_N$ satisfy \eqref{eq: condition on the mesh}.
    \begin{itemize}
        \item[1)] For any $r\ge 1$, there exists a random variable $\mathcal C_2$ that does not depend on the partition such that
        \[
            \sup_{t\in[0,T]} |Y(t) - \widehat{Y}(t)|^r \le \mathcal C_2 \Delta_N^{\lambda r} \quad a.s.
        \]
        \item[2)] If, additionally,
        \begin{equation*}
            \mathbb E\left[ \Lambda^{\frac{r(p+\max\{p, \gamma\lambda + \lambda - 1\})}{\gamma\lambda + \lambda -1}} \right] < \infty,
        \end{equation*}
        where $\lambda$ and $\Lambda$ are from \textbf{(Z2)}, $p$ is from \textbf{(B2)} and $\gamma$ is from \textbf{(B3)}, then one can choose $\mathcal C_2$ such that $\mathbb E[\mathcal C_2] < \infty$, i.e. there exists a deterministic constant $C$ that does not depend on the partition such that
        \[
            \mathbb E\left[ \sup_{t\in[0,T]} |Y(t) - \widehat{Y}(t)|^r \right] \le C\Delta_N^{\lambda r}.
        \]
    \end{itemize}
\end{theorem}

\begin{proof}
    Fix $\omega\in\Omega$ such that $Z(\omega, t)$, $t\in[0,T]$, is H\"older continuous (for simplicity of notation, we again omit $\omega$ in the brackets) and consider an arbitrary $t\in[0,T]$. Denote
    \[
        n(t) := \max\{n=0,1,...,N~|~t \ge t_n\},
    \]
    i.e. $t\in[t_{n(t)}, t_{n(t)+1})$. Then
    \begin{align*}
        |Y(t) - \widehat Y(t)|^r & \le C\left( | Y(t) - Y(t_{n(t)}) |^r + |Y(t_{n(t)}) - \widehat Y(t_{n(t)})|^r  \right)
        \\
        &\le C\Upsilon^r (t - t_{n(t)})^{\lambda r}+  C (L_2 + \Lambda)^{\frac{p r}{\gamma\lambda + \lambda - 1}} (1+\Upsilon^r) \Delta_N^{\lambda r} 
        \\
        &\le C\left(\Upsilon^r + (1+\Upsilon^r)(L_2 + \Lambda)^{\frac{p r}{\gamma\lambda + \lambda - 1}} \right) \Delta_N^{\lambda r},
    \end{align*}
    where we used Lemma \ref{lemma: Holder continuity of Y} to estimate $| Y(t) - Y(t_{n(t)}) |^r$ and bound \eqref{eq: estimation for distance between Y and hatY} to estimate $|Y(t_{n(t)}) - \widehat Y(t_{n(t)})|^r$. Therefore
    \[
        \sup_{t\in[0,T]} |Y(t) - \widehat{Y}(t)|^r \le C\left(\Upsilon^r + (1+\Upsilon^r)(L_2 + \Lambda)^{\frac{p r}{\gamma\lambda + \lambda - 1}} \right) \Delta_N^{\lambda r} =: \mathcal C_2 \Delta_N^{\lambda r}.
    \]
    Finally, using the same arguments as in Lemma \ref{lemma: Holder continuity of Y} and Lemma \ref{lemma: backward approximations in discrete time points}, it is easy to see that the condition
    \[
        \mathbb E\left[ \Lambda^{\frac{r(p+\max\{p, \gamma\lambda + \lambda - 1\})}{\gamma\lambda + \lambda -1}} \right] < \infty
    \]
    implies that 
    \[
        \mathbb E\left[ \Upsilon^r + (1+\Upsilon^r)(L_2 + \Lambda)^{\frac{p r}{\gamma\lambda + \lambda - 1}} \right] < \infty,
    \]
    therefore
    \[
        \mathbb E\left[\sup_{t\in[0,T]} |Y(t) - \widehat{Y}(t)|^r\right] \le C\left(\Upsilon^r + (1+\Upsilon^r)(L_2 + \Lambda)^{\frac{p r}{\gamma\lambda + \lambda - 1}} \right) \Delta_N^{\lambda r}
    \]
    for some constant $C>0$ that does not depend on the partition.
\end{proof}

\begin{theorem}\label{th: approx of Y-1}\hfill
    \begin{itemize}
        \item[1)] Let $Z$ satisfy \textbf{(Z1)}--\textbf{(Z2)}, Assumptions \ref{assum: B} hold and the mesh of the partition $\Delta_N$ satisfy \eqref{eq: condition on the mesh}. Then, for any $r\ge 1$, there exists a random variable $\mathcal C_3$ that does not depend on the partition such that
        \[
            \sup_{n = 0,1,...,N}\left|\frac{1}{Y(t_n) - \varphi(t_n)} - \frac{1}{\widehat{Y}(t_n) - \varphi(t_n)}\right|^r \le \mathcal C_3 \Delta_N^{\lambda r} \quad a.s.
        \]
        and
        \[
            \sup_{n = 0,1,...,N}\left|\frac{1}{\psi(t_n) - Y(t_n)} - \frac{1}{\psi(t_n) - \widehat{Y}(t_n)}\right|^r \le \mathcal C_3 \Delta_N^{\lambda r} \quad a.s.
        \]

        \item[2)] If, additionally, 
        \begin{equation}\label{eq: condition on moments of Lambda 2}
            \mathbb E\left[ \Lambda^{\frac{r(2 + p + \max\{p, \gamma\lambda + \lambda - 1\})}{\gamma\lambda + \lambda -1}} \right] < \infty,
        \end{equation}
        where $\lambda$ and $\Lambda$ are from \textbf{(Z2)}, $p$ is from \textbf{(B2)} and $\gamma$ is from \textbf{(B3)}, then one can choose $\mathcal C_3$ such that $\mathbb E[\mathcal C_3] < \infty$, i.e. there exists a deterministic constant $C$ that does not depend on the partition such that
        \[
            \mathbb E\left[ \sup_{n = 0,1,...,N}\left|\frac{1}{Y(t_n) - \varphi(t_n)} - \frac{1}{\widehat{Y}(t_n) - \varphi(t_n)}\right|^r \right] \le C\Delta_N^{\lambda r}
        \]
        and
        \[
            \mathbb E\left[ \sup_{n = 0,1,...,N}\left|\frac{1}{\psi(t_n) - Y(t_n)} - \frac{1}{\psi(t_n) - \widehat{Y}(t_n)}\right|^r \right] \le C\Delta_N^{\lambda r}.
        \]
    \end{itemize}
\end{theorem}

\begin{proof}
    By Remark \ref{rem: L1, L2 and alpha can be chosen jointly} and estimate \eqref{eq: estimation for distance between Y and hatY}, with probability 1 for any $n=0,...,N$:
    \begin{align*}
        \left|\frac{1}{Y(t_n) - \varphi(t_n)} - \frac{1}{\widehat{Y}(t_n) - \varphi(t_n)}\right|^r &= \frac{|Y(t_n) - \widehat Y(t_n)|^r}{(Y_{t_n} - \varphi(t_n))^r(\widehat{Y}_{t_n} - \varphi(t_n))^r}
        \\
        & \le \frac{(L_2+ \Lambda)^{\frac{2r}{\gamma\lambda + \lambda - 1}}}{L_1^{2r}} \sup_{n=0,1,...,N} |Y(t_n) - \widehat Y(t_n)|^r
        \\
        & \le C (L_2+ \Lambda)^{\frac{2r}{\gamma\lambda + \lambda - 1}} \xi^{pr}(1+\Upsilon^r) \Delta_N^{\lambda r}
        \\
        & =: \mathcal C_3 \Delta_N^{\lambda r}.
    \end{align*}
    It remains to notice that, by \eqref{eq: distance to pieces of bread denoted by xi} and \eqref{eq: definition of Upsilon}, the condition \eqref{eq: condition on moments of Lambda 2} implies that $\mathbb E[\mathcal C_3] < \infty$. The second estimate can be obtained in a similar manner. 
\end{proof}

\section{One-sided sandwich case}\label{sec: one-sided case}

The drift-implicit Euler approximation scheme described in Section \ref{sec: two-sided case} for the two-sided sandwich can also be adapted for the one-sided setting that corresponds to Assumptions \ref{assum: A} on the SDE \eqref{eq: volatility introduction}. However, in the two-sided sandwich case the process $Y$ was bounded (which was utilized, e.g., in Lemma \ref{lemma: Holder continuity of Y}) and, moreover, the behaviour of $Y$ was similar near both $\varphi$ and $\psi$ so that it was sufficient to analyze only one of the bounds. In the one-sided case, each $Y(t)$, for $t\in[0,T]$, is not a bounded random variable, therefore the approach from Section \ref{sec: two-sided case} has to be adjusted. For this, we will be using the inequalities \eqref{eq: upper and lower bounds for one sandwiched volatility}.

Let the noise $Z$ satisfy \textbf{(Z1)}--\textbf{(Z2)}, $Y(0)$ and $b$ satisfy Assumptions \ref{assum: A} and $Y = \{Y(t),~t \in [0,T]\}$ be the unique solution of the SDE \eqref{eq: sandwiched SDE, general form}. In line with Section \ref{sec: two-sided case}, we consider a uniform partition $\{0=t_0 < t_1<...<t_N=T\}$ of $[0,T]$, $t_k := \frac{Tk}{N}$, $k=0,1,..., N$, with the mesh $\Delta_N:=\frac{T}{N}$ such that
\begin{equation}\label{eq: condition on the mesh one-sided}
    c_3 \Delta_N < 1,
\end{equation}
where $c_3$ is an upper bound for $\frac{\partial b}{\partial y}$ from assumption \textbf{(A4)}. The backward Euler approximation $\widehat Y(t)$ is defined in a manner similar to \eqref{eq: definition of backward Euler scheme}, i.e.
\begin{equation}\label{eq: definition of backward Euler scheme one-sided}
\begin{aligned}
    \widehat Y(0) &= Y(0),
    \\
    \widehat Y(t_{k+1}) &= \widehat Y(t_{k}) + b(t_{k+1}, \widehat Y(t_{k+1}))\Delta_N + (Z(t_{k+1}) - Z(t_k)),
    \\
    \widehat Y(t) &= \widehat Y(t_k), \quad t\in[t_k, t_{k+1}),
\end{aligned}
\end{equation}
where the second expression is considered as an equation with respect to $\widehat Y(t_{k+1})$. 

\begin{remark}
    Just as in the two-sided sandwich case, each $\widehat Y(t_k)$, $k=1,...,N$, is well defined since the equation
    \[
        y - b(t, y)\Delta_N = z
    \]
    has a unique solution w.r.t. $y$ such that $y > \varphi(t)$ for any fixed $t\in[0,T]$ and any $z\in\mathbb R$. To understand this, note that assumption \textbf{(A4)} together with \eqref{eq: condition on the mesh one-sided} imply that 
    \begin{equation}\label{proofeq: hat Y exists 1}
        (y - b(t, y)\Delta_N)'_y > 0.
    \end{equation}
    Second, by \textbf{(A3)},
    \begin{gather}\label{proofeq: hat Y exists 2}
        y - b(t, y)\Delta_N \to -\infty, \quad y \to \varphi(t)+.
    \end{gather}
    Next, by \textbf{(A2)}, for any $(s,y_1)$, $(s,y_2) \in \overline{\mathcal D_1} := \{(u,y)\in [0,T]\times \mathbb R_+, y\in [\varphi(u) + 1, \infty)\}$ we have that
    \[
        |b(s,y_1) - b(s,y_2)| \le c_1|y_1-y_2|,
    \]
    i.e.
    \[
        \sup_{(s,y)\in \overline{\mathcal D_1}} \left|\frac{\partial b}{\partial y} (s, y)\right| < \infty.
    \]
    Using this, \textbf{(A4)} and the mean value theorem, for any positive $y \ge \varphi(t) +1$
    \begin{align*}
        b(t,y) &= b(t,\varphi(t) + 1) + \frac{\partial b}{\partial y} (t, \theta_y) (y-1-\varphi(t)) 
        \\
        &\le \max_{s\in[0,T]} b(t,\varphi(t) + 1) + \max_{s\in[0,T]}|1+\varphi(s)| \sup_{(s,y)\in \overline{\mathcal D_1}} \left|\frac{\partial b}{\partial y} (s, y)\right| + c_3 y
        \\
        &=: C + c_3y,
    \end{align*}
    whence
    \begin{equation}\label{proofeq: hat Y exists 3}
        y - b(t,y)\Delta_N \ge -C\Delta_N + (1 - c_3\Delta_N) y \to \infty, \quad y\to\infty.
    \end{equation}
    Existence and uniqueness of the solution then follows from \eqref{proofeq: hat Y exists 1}--\eqref{proofeq: hat Y exists 3}.
\end{remark}

\begin{remark}
    Similarly to the two-sided sandwich case, the value of $\widehat Y(t)$ for $t\in[0,T]\setminus\{t_0,...,t_N\}$ can also be defined via linear interpolation with no changes in formulations of the results and almost no variations in the proofs.
\end{remark}

Our strategy for proving the convergence of $\widehat Y$ to $Y$ will be similar to what we have done in section \ref{sec: two-sided case}. Therefore we will be omitting the details highlighting only the points which are different from the two-sided sandwich case. We start with some useful properties of $\widehat Y$ and $Y$.

\begin{lemma}\label{lemma: bounds for discretized process one sided}
    Let $Z$ satisfy \textbf{(Z1)}--\textbf{(Z2)}, Assumptions \ref{assum: A} hold and the mesh of the partition $\Delta_N$ satisfy \eqref{eq: condition on the mesh one-sided}. Then there exist deterministic constants $L_1$, $L_2 > 0$ depending only on $Y(0)$, the shape of the drift $b$ and $\lambda$, such that
    \[
        \widehat Y(t_n) \ge \varphi(t_n) + \frac{L_1}{ ( L_2 + \Lambda )^{ \frac{1}{\gamma\lambda + \lambda - 1} } } \quad a.s.,
    \]
    where $\Lambda$ is from assumption \textbf{(Z2)} and $\gamma$ is from assumption \textbf{(A3)}. Moreover, there exist constants $L_3$, $L_4 > 0$ that also depend only on $Y(0)$, the shape of the drift $b$ and $\lambda$ such that
    \[
        \widehat Y(t_n) \le L_3 + L_4\Lambda, \quad n = 0, 1,... ,N, \quad a.s.
    \]
    for all partitions with the mesh satisfying $\frac{c_1 }{(Y(0) - \varphi(0))^p}\Delta_N < 1$ with $c_1$ and $p$ being from \textbf{(A2)}.
\end{lemma}

\begin{proof}
    The proof of
    \[
        \widehat Y(t_n) \ge \varphi(t_n) + \frac{L_1}{ ( L_2 + \Lambda )^{ \frac{1}{\gamma\lambda + \lambda - 1} } }
    \]
    is identical to the corresponding one in Lemma \ref{lemma: bounds for discretized process} and will be omitted. Let us prove that
    \[
        \widehat Y(t_n) \le L_3 + L_4\Lambda \quad a.s.
    \]
    Fix $\omega\in\Omega$ for which $Z(\omega, t)$ is H\"older continuous, consider a partition with the mesh satisfying $\frac{c_1 }{(Y(0) - \varphi(0))^p}\Delta_N < 1$ and fix an arbitrary $n=0,1,...,N-1$. Assume that $\widehat Y(t_{n+1}) > \varphi(t_{n+1}) + (Y(0) - \varphi(0))$ (otherwise the claim of the lemma holds automatically). Put
    \[
        \kappa(n):= \max\{k=0,1,...,n~|~\widehat Y(t_k) \le \varphi(t_k) + (Y(0) - \varphi(0)) \}
    \]
    and observe that $(t_k, \widehat Y(t_k)) \in \mathcal D_{Y(0) - \varphi(0)}$ for any $k=\kappa(n)+1,..., n+1$, where $\mathcal D_{Y(0) - \varphi(0)}$ is defined via \eqref{eq: definition of the set D}. Next, by \textbf{(A2)}, for any $y\in\mathcal D_{Y(0) - \varphi(0)}$
    \begin{align*}
        |b(t,y)| - \left|b\Big(t, \varphi(t) + (Y(0) - \varphi(0))\Big)\right| &\le \left|b(t,y) - b\Big(t, \varphi(t) + (Y(0) - \varphi(0))\Big)\right| 
        \\
        &\le \frac{c_1}{(Y(0) - \varphi(0))^p} |y - \varphi(t) - (\varphi(0) - Y(0))| 
        \\
        &\le \frac{c_1}{(Y(0) - \varphi(0))^p}|y| + \frac{c_1}{(Y(0) - \varphi(0))^p}|\varphi(t) + (\varphi(0) - Y(0))|,
    \end{align*}
    i.e. there exists a constant $C>0$ that does not depend on the partition such that
    \begin{equation}\label{proofeq: linear growth of b one sided}
        |b(t,y)| \le C +  \frac{c_1}{(Y(0) - \varphi(0))^p} |y|.
    \end{equation}
    Next, observe that, for any $k=\kappa(n)+1,..., n+1$, we have
    \begin{align*}
        \widehat Y(t_k) &= \widehat Y(t_{\kappa(n)}) + \sum_{i=\kappa(n) + 1}^k b(t_i, \widehat Y(t_i)) \Delta_N + Z(t_k) - Z(t_{\kappa(n)})
        \\
        &\le \varphi(t_{\kappa(n)}) + (Y(0) - \varphi(0)) + \sum_{i=\kappa(n) + 1}^k b(t_i, \widehat Y(t_i)) \Delta_N + \Lambda (t_k - t_{\kappa(n)})^\lambda
        \\
        &\le \left|\max_{s\in[0,T]} \varphi(s) + (Y(0) - \varphi(0))\right| + T^\lambda\Lambda + \sum_{i=\kappa(n) + 1}^k b(t_i, \widehat Y(t_i)) \Delta_N.
    \end{align*}
    Therefore, using \eqref{proofeq: linear growth of b one sided} and
    \[
        \widehat Y(t_k) > \varphi(t_k) \ge \min_{s\in[0,T]} \varphi(s),
    \]
    one can write
    \begin{align*}
        |\widehat Y(t_k)| & \le \left|\min_{s\in[0,T]} \varphi(s)\right| + \left|\max_{s\in[0,T]} \varphi(s) + (Y(0) - \varphi(0))\right| + T^\lambda\Lambda + \sum_{i=\kappa(n) + 1}^k |b(t_i, \widehat Y(t_i))| \Delta_N
        \\
        &\le \left|\min_{s\in[0,T]} \varphi(s)\right| + \left|\max_{s\in[0,T]} \varphi(s) + (Y(0) - \varphi(0))\right| + T^\lambda\Lambda +  C\sum_{i=\kappa(n) + 1}^k \Delta_N 
        \\
        &\qquad + \frac{c_1}{(Y(0) - \varphi(0))^p}\sum_{i=\kappa(n) + 1}^k  |\widehat Y(t_i)| \Delta_N,
    \end{align*}
    where $C>0$ is some positive constant that does not depend on the partition.
    
    Now we want to apply the discrete version of the Gronwall inequality from \cite[Lemma A.3]{Kruse_2014}. In order to do that, we observe that
    \begin{align*}
        |\widehat Y(t_{\kappa(n)+1})| &\le C + T^\lambda\Lambda + \frac{c_1}{(Y(0) - \varphi(0))^p} \Delta_N |\widehat Y(t_k)|,
    \end{align*}
    and, for any $k=\kappa(n)+2,..., n+1$,
    \begin{align*}
        |\widehat Y(t_k)| &\le C + T^\lambda\Lambda + \frac{c_1}{(Y(0) - \varphi(0))^p}\sum_{i=\kappa(n) + 1}^{k-1}  |\widehat Y(t_i)| \Delta_N + \frac{c_1}{(Y(0) - \varphi(0))^p} \Delta_N |\widehat Y(t_k)|.
    \end{align*}
    Now, since $\frac{c_1}{(Y(0) - \varphi(0))^p} \Delta_N < 1$, we can write that 
    \[
        \left( 1 - \frac{c_1}{(Y(0) - \varphi(0))^p} \Delta_N \right) |\widehat Y(t_{\kappa(n)+1})| \le C + T^\lambda\Lambda
    \]
    and, for all $k=\kappa(n)+2,..., n+1$,
    \[
        \left( 1 - \frac{c_1}{(Y(0) - \varphi(0))^p} \Delta_N \right) |\widehat Y(t_k)| \le C + T^\lambda\Lambda + \frac{c_1}{(Y(0) - \varphi(0))^p}\sum_{i=\kappa(n) + 1}^{k-1}  |\widehat Y(t_i)| \Delta_N.
    \]
    Put 
    \[
        N_0 := \min\left\{N\ge 1:~\frac{c_1 }{(Y(0) - \varphi(0))^p}\Delta_N < 1\right\} = \left[ \frac{T c_1}{(Y(0) - \varphi(0))^p} \right] + 1 
    \]
    with $[x]$ being the greatest integer less than or equal to $x$ and observe that, for all $N\ge N_0$,
    \[
        1 - \frac{c_1 }{(Y(0) - \varphi(0))^p}\Delta_N \ge 1 - \frac{c_1 }{(Y(0) - \varphi(0))^p}\Delta_{N_0}.
    \]
    Therefore, 
    \begin{align*}
        |\widehat Y(t_{\kappa(n)+1})| &\le \frac{C}{1 - \frac{c_1}{(Y(0) - \varphi(0))^p} \Delta_N} + \frac{T^\lambda}{1 - \frac{c_1}{(Y(0) - \varphi(0))^p} \Delta_N}\Lambda 
        \\
        & \le \frac{C}{1 - \frac{c_1}{(Y(0) - \varphi(0))^p} \Delta_{N_0}} + \frac{T^\lambda}{1 - \frac{c_1}{(Y(0) - \varphi(0))^p} \Delta_{N_0}}\Lambda
        \\
        & =: C_1 + C_2 \Lambda
    \end{align*}
    and, for all $k=\kappa(n)+2,..., n+1$,
    \begin{align*}
        |\widehat Y(t_k)| &\le \frac{C}{1 - \frac{c_1}{(Y(0) - \varphi(0))^p} \Delta_N} + \frac{T^\lambda}{1 - \frac{c_1}{(Y(0) - \varphi(0))^p} \Delta_N}\Lambda 
        \\
        &\quad+ \frac{c_1}{(Y(0) - \varphi(0))^p}\sum_{i=\kappa(n) + 1}^{k-1}  |\widehat Y(t_i)| \frac{\Delta_N}{1 - \frac{c_1}{(Y(0) - \varphi(0))^p} \Delta_N}
        \\
        & \le \frac{C}{1 - \frac{c_1}{(Y(0) - \varphi(0))^p} \Delta_{N_0}} + \frac{T^\lambda}{1 - \frac{c_1}{(Y(0) - \varphi(0))^p} \Delta_{N_0}}\Lambda 
        \\
        &\quad+ \frac{c_1}{(Y(0) - \varphi(0))^p}\sum_{i=\kappa(n) + 1}^{k-1}  |\widehat Y(t_i)| \frac{\Delta_N}{1 - \frac{c_1}{(Y(0) - \varphi(0))^p} \Delta_{N_0}}
        \\
        & =: C_1 + C_2 \Lambda + C_3 \sum_{i=\kappa(n) + 1}^{k-1}  |\widehat Y(t_i)| \Delta_N.
    \end{align*}
    Using a discrete version of the Gronwall inequality, we now obtain that for all $k=\kappa(n)+1,..., n+1$
    \begin{align*}
        |\widehat Y(t_k)| &\le \left(C_1 + C_2 \Lambda\right) \exp\left\{ C_3 \sum_{i=\kappa(n) + 1}^{k-1} \Delta_N \right\} \le \left(C_1 + C_2 \Lambda\right) \exp\left\{ T C_3 \right\}
        \\
        &=: L_3 + L_4 \Lambda.
    \end{align*}
    which ends the proof.
\end{proof}

\begin{remark}\label{rem: L1, L2 and alpha can be chosen jointly one sided}
    It is clear that constants $L_1$, $L_2$, $L_3$ and $L_4$ can be chosen jointly for $Y$ and $\widehat Y$, so that the inequalities
    \[
        \varphi(t) + \frac{L_1}{ ( L_2 + \Lambda )^{\alpha} } \le Y(t) \le L_3 + L_4 \Lambda, \quad t\in [0,T],
    \]
    and
    \[
        \varphi(t_n) + \frac{L_1}{ ( L_2 + \Lambda )^{\alpha} } \le \widehat Y(t_n) \le L_3 + L_4 \Lambda, \quad n = 0, 1,... ,N,
    \]
    hold simultaneously with probability 1. 
\end{remark}

Next, corresponding to Lemma \ref{lemma: Holder continuity of Y} in the two-sided case, $Y$ enjoys H\"older continuity with the H\"older constant being integrable provided that $\Lambda$ has moments of sufficiently high order. This is summarized in the lemma below.

\begin{lemma}\label{lemma: Holder continuity of Y one sided}
    Let $Z$ satisfy \textbf{(Z1)}--\textbf{(Z2)} and assumptions \textbf{(A1)}--\textbf{(A3)} hold. 
    \begin{itemize}
        \item[1)] There exists a positive random variable $\Upsilon$ such that with probability 1
        \[
            |Y(t) - Y(s)| \le \Upsilon |t-s|^\lambda, \quad t, s\in[0,T].
        \]
        \item[2)] If, for some $ r \ge 1$,
        \begin{equation}\label{eq: condition on Lambda to put expectation one sided}
            \mathbb E \left[\Lambda^{\frac{r(p + \gamma\lambda + \lambda - 1)}{\gamma\lambda + \lambda - 1}}\right] < \infty,
        \end{equation}
        where $\lambda$ and $\Lambda$ are from \textbf{(Z2)}, $p$ is from \textbf{(A2)} and $\gamma$ is from \textbf{(A3)}, then one can choose $\Upsilon$ such that
        \[
            \mathbb E [\Upsilon^{r}] < \infty.
        \]
    \end{itemize}
\end{lemma}
\begin{proof}
    By \eqref{eq: upper and lower bounds for one sandwiched volatility}, 
    \[
        Y(t) \ge \varphi(t) + \frac{L_1}{ ( L_2 + \Lambda )^{\frac{1}{\gamma \lambda + \lambda - 1}} } \quad a.s.,
    \]
    i.e. with probability 1 $(t,Y(t)) \in \mathcal D_{\frac{1}{\xi}}$, $t\in[0,T]$, where
    \begin{equation}\label{eq: distance to pieces of bread denoted by xi one sided}
        \xi := \frac{( L_2 + \Lambda )^{\frac{1}{\gamma\lambda + \lambda - 1}}}{L_1}
    \end{equation}
    and $\mathcal D_{\frac 1 \xi}$ is defined in \eqref{eq: definition of the set D}. Denote $\phi(t) := \varphi(t) + 1$ and notice that $(t,\phi(t)) \in \mathcal D_{\frac{1}{\xi}}$, $t\in[0,T]$, since $\frac{1}{\xi} \le Y(0) - \varphi(0)$. Thus, using the same arguments as applied in \eqref{proofeq: estimate for Holder cont two sided}, we can write that, with probability 1, for any $0\le s < t \le T$:
    \[
        |Y(t) - Y(s)| \le c_1 \xi^p \int_s^t |Y(u) - \phi(u)| du + \max_{u\in[0,T]}|b(u, \phi(u))| (t-s) + \Lambda (t - s)^\lambda,
    \]
    where $c_1$ is from \textbf{(A2)}. Now, again by \eqref{eq: upper and lower bounds for one sandwiched volatility},
    \[
        Y(t) \le L_3 + L_4\Lambda \quad a.s., 
    \]
    hence with probability 1
    \begin{align*}
        |Y(t) - Y(s)| &\le c_1 \xi^p \int_s^t |Y(u) - \phi(u)| du + \max_{u\in[0,T]}|b(u, \phi(u))| (t-s) + \Lambda (t - s)^\lambda
        \\
        &\le c_1 \xi^p (L_3 + L_4\Lambda)(t-s) + c_1 \xi^p\max_{u\in[0,T]}|\phi(u)| (t-s) 
        \\
        &\quad+ \max_{u\in[0,T]}|b(u, \phi(u))| (t-s) + \Lambda (t - s)^\lambda
        \\
        &\le C(1+\xi^p\Lambda + \xi^p + \Lambda) (t-s)^\lambda,
    \end{align*}
    where $C$ is a positive constant. Now one can put
    \begin{equation}\label{eq: definition of Upsilon one sided}
        \Upsilon := C(1+\xi^p\Lambda + \xi^p + \Lambda)
    \end{equation}
    and observe that
    \[
        \mathbb E [\Upsilon^{r}] < \infty
    \]
    whenever \eqref{eq: condition on Lambda to put expectation one sided} holds.
\end{proof}

\begin{corollary}\label{rem: Holder approximations one sided}
    Using Lemma \ref{lemma: bounds for discretized process one sided} and following the proof of Lemma \ref{lemma: Holder continuity of Y one sided}, it is easy to obtain that, for any partition with the mesh satisfying
    \begin{equation}\label{eq: ultimate condition on partition one sided}
        \max\left\{c_3, \frac{c_1 }{(Y(0) - \varphi(0))^p}\right\} \Delta_N < 1
    \end{equation}
    there is a random variable $\Upsilon$ independent of the partition such that with probability 1
    \begin{equation}\label{eq: pseudo Holder continuity of discretised process one sided}
        |\widehat Y(t_k) - \widehat Y(t_n)| \le \Upsilon |t_k - t_n|^\lambda, \quad k,n = 0,...,N.
    \end{equation}
    Furthermore, just like in Lemma \ref{lemma: Holder continuity of Y}, for $r>0$
    \[
        \mathbb E [\Upsilon^r] < \infty
    \]
    provided that
    \[
        \mathbb E \left[\Lambda^{\frac{r(p + \gamma\lambda + \lambda - 1)}{\gamma\lambda + \lambda - 1}}\right] < \infty.
    \]
    Finally, such $\Upsilon$ can be chosen jointly for $Y$ and $\widehat Y$, so that
    \[
        |Y(t) - Y(s)| \le \Upsilon |t-s|^\lambda, \quad t,s\in[0,T],
    \]
    holds simultaneously with \eqref{eq: pseudo Holder continuity of discretised process one sided} with probability 1.
\end{corollary}

\begin{lemma}\label{lemma: backward approximations in discrete time points one sided}

    Let $Z$ satisfy \textbf{(Z1)}--\textbf{(Z2)}, Assumptions \ref{assum: A} hold and the mesh of the partition $\Delta_N$ satisfy \eqref{eq: condition on the mesh one-sided}.
    \begin{itemize}
        \item[1)] For any $r\ge 1$, there exists a positive random variable $\mathcal C_4$ that does not depend on the partition such that
        \[
            \sup_{k=0,1,...,N} |Y(t_k) - \widehat{Y}(t_k)|^r \le \mathcal C_4 \Delta_N^{\lambda r}  \quad a.s. 
        \]
        \item[2)] If, additionally,
        \begin{equation}\label{eq: condition on moments of Lambda one sided}
            \mathbb E\left[ \Lambda^{\frac{r(2p+ \gamma\lambda + \lambda - 1)}{\gamma\lambda + \lambda -1}} \right] < \infty,
        \end{equation}
        where $\lambda$ and $\Lambda$ are from \textbf{(Z2)}, $p$ is from \textbf{(A2)} and $\gamma$ is from \textbf{(A3)}, then one can choose $\mathcal C_4$ such that  $\mathbb E [\mathcal C_4] < \infty$, i.e. there exists a deterministic constant $C$ that does not depend on the partition such that
        \[
            \mathbb E\left[\sup_{k=0,1,...,N} |Y(t_k) - \widehat{Y}(t_k)|^r \right] \le C\Delta_N^{\lambda r}.
        \]
    \end{itemize}
\end{lemma}

\begin{proof}
    Following the proof of Lemma \ref{lemma: backward approximations in discrete time points}, one can easily obtain that for any $n=0,1,...,N$
    \[
        |Y(t_n) - \widehat Y(t_n)| \le C\left(\sum_{i=1}^n \int_{t_{i-1}}^{t_{i}}\left| b(s,Y(s)) - b(t_i, Y(t_i))\right|  ds\right)^r.
    \]
    Next, note that $(t, Y(t)) \in \mathcal D_{\frac 1 \xi}$, where $\xi$ is defined by \eqref{eq: distance to pieces of bread denoted by xi one sided}, so, by \textbf{(A2)} and Lemma \ref{lemma: Holder continuity of Y one sided},
    \begin{align*}
        &\left(\sum_{i=1}^n \int_{t_{i-1}}^{t_{i}}\left| b(s,Y(s)) - b(t_i, Y(t_i))\right|  ds\right)^r
        \\
        & \le C\xi^{pr}\left(\sum_{i=1}^n \int_{t_{i-1}}^{t_{i}}\left| s - t_i \right|^\lambda  ds\right)^r + C\xi^{pr} \left(\sum_{i=1}^n \int_{t_{i-1}}^{t_{i}}\left| Y(s) - Y(t_i)\right|  ds\right)^r
        \\
        &\le C\xi^{pr}\left(\sum_{i=1}^n \int_{t_{i-1}}^{t_{i}}\left| s - t_i \right|^\lambda  ds\right)^r + C\xi^{pr}\Upsilon^r \left(\sum_{i=1}^n \int_{t_{i-1}}^{t_{i}}\left| s - t_i\right|^\lambda  ds\right)^r
        \\
        & = C\xi^{pr}(1+\Upsilon^r)\left(\sum_{i=1}^n \frac{1}{(1+\lambda)} \Delta_N^{1+\lambda} \right)^r
        \\
        & \le C\xi^{pr}(1+ \Upsilon^r) \Delta_N^{\lambda r},
    \end{align*}
    i.e.
    \begin{equation}\label{eq: estimation for distance between Y and hatY one sided}
        \sup_{n = 0,...,N}|Y(t_n) - \widehat Y(t_n)|^r \le C\xi^{pr}(1+\Upsilon^r) \Delta_N^{\lambda r}.
    \end{equation}
    In order to conclude the proof, it remains to notice that \eqref{eq: distance to pieces of bread denoted by xi one sided}, \eqref{eq: definition of Upsilon one sided} and \eqref{eq: condition on moments of Lambda one sided} imply that
    \[
        \mathbb E\left[\xi^{pr}(1+\Upsilon^r)\right] < \infty.
    \]
\end{proof}

Now we are ready to formulate the two main results of this section.

\begin{theorem}\label{th: main result for one sided sandwich}
    Let $Z$ satisfy \textbf{(Z1)}--\textbf{(Z2)}, Assumptions \ref{assum: A} hold and the mesh of the partition $\Delta_N$ satisfy \eqref{eq: ultimate condition on partition one sided}.
    \begin{itemize}
        \item[1)] For any $r\ge 1$, there exists a random variable $\mathcal C_5$ that does not depend on the partition such that
        \[
            \sup_{t\in[0,T]} |Y(t) - \widehat{Y}(t)|^r \le \mathcal C_5 \Delta_N^{\lambda r} \quad a.s.
        \]
        \item[2)] If, additionally,
        \begin{equation*}
            \mathbb E\left[ \Lambda^{\frac{r(2p+\gamma\lambda + \lambda - 1)}{\gamma\lambda + \lambda -1}} \right] < \infty,
        \end{equation*}
        where $\lambda$ and $\Lambda$ are from \textbf{(Z2)}, $p$ is from \textbf{(A2)} and $\gamma$ is from \textbf{(A3)}, then one can choose $\mathcal C_5$ such that $\mathbb E[\mathcal C_5] < \infty$, i.e. there exists a deterministic constant $C$ that does not depend on the partition such that
        \[
            \mathbb E\left[ \sup_{t\in[0,T]} |Y(t) - \widehat{Y}(t)|^r \right] \le C\Delta_N^{\lambda r}.
        \]
    \end{itemize}
\end{theorem}
\begin{proof}
    The proof is similar to the one of Theorem \ref{th: main result for double sandwich} but instead of Lemmas \ref{lemma: Holder continuity of Y}, \ref{lemma: backward approximations in discrete time points} and bound \eqref{eq: estimation for distance between Y and hatY} one should apply Lemmas \ref{lemma: Holder continuity of Y one sided}, \ref{lemma: backward approximations in discrete time points one sided} and bound \eqref{eq: estimation for distance between Y and hatY one sided}.
\end{proof}

\begin{theorem}
    Let $Z$ satisfy \textbf{(Z1)}--\textbf{(Z2)}, Assumptions \ref{assum: A} hold and the mesh of the partition $\Delta_N$ satisfy \eqref{eq: ultimate condition on partition one sided}. 
    \begin{itemize}
        \item[1)] For any $r\ge 1$, there exists a random variable $\mathcal C_6$ that does not depend on the partition such that
        \[
            \sup_{n = 0,1,...,N}\left|\frac{1}{Y(t_n) - \varphi(t_n)} - \frac{1}{\widehat{Y}(t_n) - \varphi(t_n)}\right|^r \le \mathcal C_6 \Delta_N^{\lambda r} \quad a.s.
        \]

        \item[2)] If, additionally, 
        \begin{equation}\label{eq: condition on moments of Lambda one sided 2}
            \mathbb E\left[ \Lambda^{\frac{r(2+2p+ \gamma\lambda + \lambda - 1)}{\gamma\lambda + \lambda -1}}  \right] < \infty,
        \end{equation}
        where $\lambda$ and $\Lambda$ are from \textbf{(Z2)}, $p$ is from \textbf{(A2)} and $\gamma$ is from \textbf{(A3)}, then one can choose $\mathcal C_6$ such that $\mathbb E[\mathcal C_6] < \infty$, i.e. there exists a deterministic constant $C$ that does not depend on the partition such that
        \[
            \mathbb E\left[ \sup_{n = 0,1,...,N}\left|\frac{1}{Y(t_n) - \varphi(t_n)} - \frac{1}{\widehat{Y}(t_n) - \varphi(t_n)}\right|^r \right] \le C\Delta_N^{\lambda r}.
        \]
    \end{itemize}
\end{theorem}
\begin{proof}
    The proof is similar to Theorem \ref{th: approx of Y-1} and is omitted.
\end{proof}

\section{Examples and simulations}\label{sec: examples}

The algorithms presented in \eqref{eq: definition of backward Euler scheme} and \eqref{eq: definition of backward Euler scheme one-sided} imply that, in order to generate $\widehat Y(t_{n+1})$, one has to solve an equation that potentially can be challenging from the computational point of view. However, in some cases that are relevant for applications this equation has a simple explicit solution.

Regarding the numerical examples that follow, we remark that:
\begin{itemize}
    \item[1)] all the simulations are performed in the \textsf{R} programming language on the system with Intel Core i9-9900K CPU and 64 Gb RAM;
        
    \item[2)] in order to simulate paths of fractional Brownian motion, \textsf{R} package \textsf{somebm} is used;
        
    \item[3)] in Example \ref{ex: sandwich mBm sim}, discrete samples of the multifractional Brownian motion (mBm) values are simulated using the Cholesky decomposition of the corresponding covariance matrix (for covariance structure of the mBm, see e.g. \cite[Proposition 4]{ACLV2002}) and the \textsf{R} package \textsf{nleqslv} is used for solving \eqref{eq: definition of backward Euler scheme} numerically.
\end{itemize}

\begin{example}\label{ex: CIR sim}{(\emph{Generalized CIR processes})}
    Let $\varphi \equiv 0$, $Z$ satisfy \textbf{(Z1)}--\textbf{(Z2)} with $\lambda \in\left(\frac{1}{2},1\right)$, $Y(0)$, $\kappa_1$, $\kappa_2 >0$, $\gamma > \frac{1}{\lambda} - 1$ be given and $\{Y(t),~t\in[0,T]\}$ satisfy the SDE of the form
    \begin{equation}\label{eq: CIR sim section}
        Y(t) = Y(0) + \int_0^t \left(\frac{\kappa_1}{Y(s)} - \kappa_2 Y(s)\right)ds + Z(t), \quad t\in[0,T].
    \end{equation}
    This process fits into the framework of Section \ref{sec: one-sided case} and the equation for $\hat Y(t_{k+1})$ from \eqref{eq: definition of backward Euler scheme one-sided} reads as follows:
    \[
        \widehat Y(t_{k+1}) = \widehat Y(t_{k}) + \left(\frac{\kappa_1}{\widehat Y(t_{k+1})} - \kappa_2 \widehat Y(t_{k+1})\right)\Delta_N + Z(t_{k+1}) - Z(t_k).
    \]
    It is easy to see that it has a unique positive solution
    \[
        \widehat Y(t_{k+1}) = \frac{ \widehat Y(t_{k}) + (Z(t_{k+1}) - Z(t_k)) + \sqrt{ \Big(\widehat Y(t_{k}) + (Z(t_{k+1}) - Z(t_k))\Big)^2 + 4\kappa_1\Delta_N(1+\kappa_2\Delta_N) } }{ 2(1+\kappa_2\Delta_N) }.
    \]
    
    Fig. \ref{fig1} contains 10 sample paths of the process \eqref{eq: CIR sim section} driven by a fractional Brownian motion with $H=0.7$. In all simulation we take $N = 10000$, $T=1$ and $Y(0) = 1 = \kappa_1 = \kappa_2 =1$. Based on 10000 simulations, the average time for simulating one path is 0.005388308 seconds.
    
    \begin{figure}[h!]
        \centering
        \includegraphics[width = 0.7\textwidth]{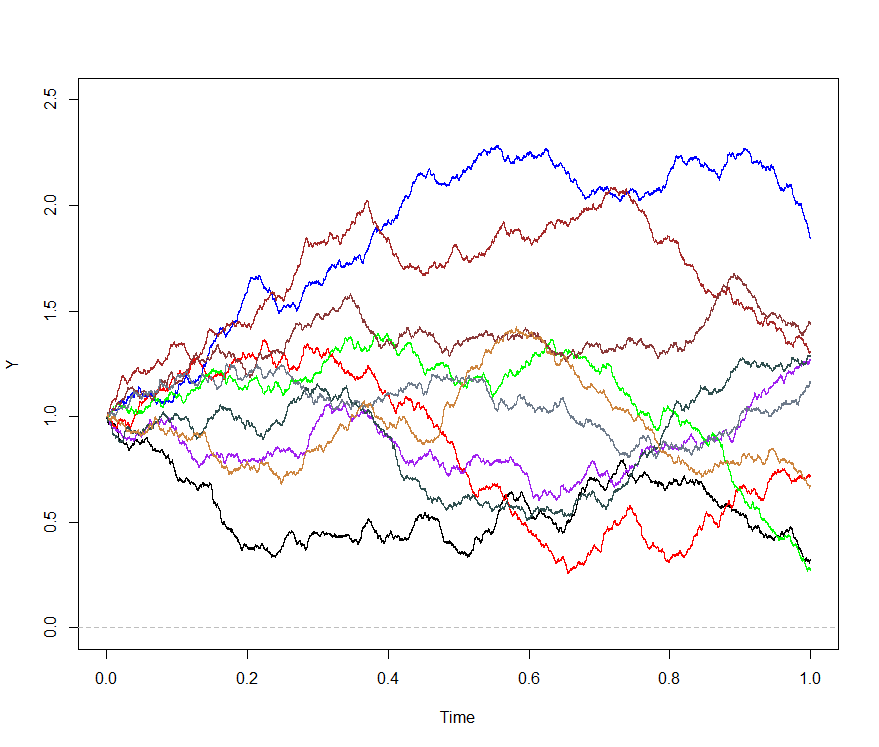}
        \caption{Ten sample paths of \eqref{eq: CIR sim section} generated using the backward Euler approximation scheme; $N = 10000$, $T=1$, $Y(0) = 1 = \kappa_1 = \kappa_2 =1$, $Z$ is a fractional Brownian motion with $H = 0.7$.}
        \label{fig1}
    \end{figure}

    Note that the drift-implicit Euler scheme for \eqref{eq: CIR sim section} driven by the fractional Brownian motion was the main subject of \cite{Hong2020} and \cite{ZhYu2020} but in both these works the convergence of $\hat Y$ to $Y$ is established only on $[0,T]$ with $T$ being small (see e.g.  \cite[Eq.  (8) and Remark 3.1]{Hong2020}). Our results fill this gap and convergence holds on arbitrary $[0,T]$ for any model parameters.
\end{example}

\begin{example}\label{ex: TSB sim}{\textit{(Sandwiched process of the TSB type)}}
    Consider a sandwiched SDE of the form
    \begin{equation}\label{eq: TSB sim}
        Y(t) = Y(0) + \int_0^t \left( \frac{\kappa_1}{Y(s) - \varphi(s)} - \frac{\kappa_2}{\psi(s) - Y(s)} - \kappa_3 Y(s) \right)ds + Z(t),\quad t\in[0,T],
    \end{equation}
    where $Z$ satisfies \textbf{(Z1)}--\textbf{(Z2)} with $\lambda \in \left(\frac{1}{2}, 1\right)$. This equation fits into the framework of Section \ref{sec: one-sided case} and the scheme \eqref{eq: definition of backward Euler scheme} leads to $N$ cubic equations of the form
    \begin{equation}\label{eq: cubic equation for simulation}
        \widehat Y^3(t_{n+1}) + B_{2, n} \widehat Y^2(t_{n+1}) + B_{1, n} \widehat Y(t_{n+1}) + B_{0, n} = 0, \quad n=0,...,N-1,
    \end{equation}
    where
    \begin{align*}
        B_{0, n} &:= \frac{-\varphi(t_{n+1}) \psi(t_{n+1})\left( \widehat Y(t_{n}) + \Delta Z_n \right) + \Delta_N \left( \kappa_1\psi(t_{n+1}) + \kappa_2\varphi(t_{n+1}) \right)}{1 + \Delta_N \kappa_3}, \\
        B_{1,n} &:= \varphi(t_{n+1}) \psi(t_{n+1}) + \frac{(\varphi(t_{n+1}) + \psi(t_{n+1}))(\widehat Y(t_{n}) + \Delta Z_n) - \Delta_N (\kappa_1 + \kappa_2)}{1 + \Delta_N \kappa_3},
        \\
        B_{2,n} &:= -\varphi(t_{n+1}) - \psi(t_{n+1}) - \frac{\widehat Y(t_{n}) + \Delta Z_n}{1 + \Delta_N \kappa_3},
    \end{align*}
    Note this equation can be solved explicitly using, e.g., the celebrated Cardano method. Namely, define
    \[
        Q_{1,n} := B_{1,n} - \frac{B_{2,n}^2}{3}, \quad Q_{2,n} := \frac{2B_{2,n}}{27} - \frac{B_{2,n}B_{1,n}}{3} + B_{0,n} 
    \]
    and put
    \[
        Q_n = \left( \frac{Q_{1,n}}{3} \right)^3 + \left( \frac{Q_{2,n}}{2} \right)^2,
    \]
    \[
        \alpha_n := \sqrt[3]{- \frac{Q_{2,n}}{2} + \sqrt{Q_n}}, \quad \beta_n := \sqrt[3]{- \frac{Q_{2,n}}{2} - \sqrt{Q_n}},
    \]
    where among possible complex values of $\alpha_n$ and $\beta_n$ one should take those for which $\alpha_n\beta_n = -\frac{Q_{1,n}}{3}$. Then the three roots of the cubic equation \eqref{eq: cubic equation for simulation} are
    \begin{align*}
        y_{1,n} &= \alpha_n + \beta_n,
        \\
        y_{2,n} &= -\frac{\alpha_n + \beta_n}{2} + i\frac{\alpha_n - \beta_n}{2}\sqrt{3},
        \\
        y_{3,n} &= -\frac{\alpha_n + \beta_n}{2} - i\frac{\alpha_n - \beta_n}{2}\sqrt{3},
    \end{align*}
    and $\widehat Y(t_{n+1})$ is equal to the root which belongs to $(\varphi(t_{n+1}), \psi(t_{n+1}))$ (note that there is exactly one root in that interval).
    
    Fig. \ref{fig2} contains 10 sample paths of the process \eqref{eq: TSB sim} driven by a fractional Brownian motion with $H=0.7$. In all simulation we take $N = 10000$, $T=1$ and $Y(0) = 0$, $\kappa_1 = \kappa_2 =\frac{1}{2}$, $\kappa_3 = 0$ (this case corresponds to the TSB equation described in Example \ref{ex: TSB}). Simulation is performed by direct implementation of the Cardano's method in \textsf{R}; based on 10000 simulations, the average time for simulating one path is 0.03700142 seconds.
    
    \begin{figure}[h!]
        \centering
        \includegraphics[width = 0.7\textwidth]{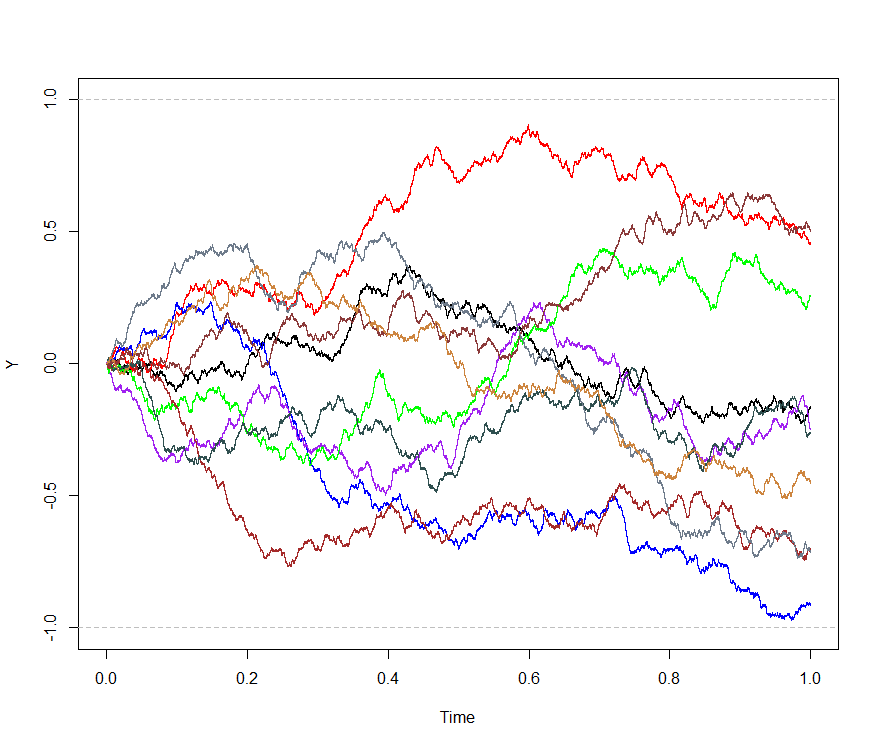}
        \caption{Ten sample paths of \eqref{eq: TSB sim} generated using the backward Euler approximation scheme; $N = 10000$, $T=1$, $Y(0) = 0$, $\kappa_1 = \kappa_2 =\frac{1}{2}$, $\kappa_3 = 0$, $Z$ is a fractional Brownian motion with $H = 0.7$.}
        \label{fig2}
    \end{figure}
    
\end{example}

In both Examples \ref{ex: CIR sim} and \ref{ex: TSB sim}, equations for computing $\widehat Y$ could be explicitly solved but the H\"older continuity of the noise could not be less then $1/2$. The next example shows that the drift-implicit Euler scheme can be applied in the rough case as well.

\begin{example}\label{ex: sandwich mBm sim}{\textit{(Sandwiched process driven by multifractional Brownian motion)}}
    Consider the sandwiched SDE of the form
    \begin{equation}\label{eq: sandwich mBm}
        Y(t) = Y(0) + \int_0^t \left(\frac{\kappa_1}{(Y(s) - \varphi(s))^4} - \frac{\kappa_2}{(\psi(s) - Y(s))^4}\right)ds + Z(t), \quad t\in[0,T].
    \end{equation}
    In this case, Theorem \ref{th: properties of sandwich} guarantees existence and uniqueness of the solution for $\lambda$-H\"older $Z$ with $\lambda > \frac{1}{5}$ (note that this equation fits the framework of Example \ref{ex: generalized TSB} from Section \ref{sec: assumptions}). On Fig. 3, one can see paths of the process \eqref{eq: sandwich mBm} with $\kappa_1 = \kappa_2 = 1$, $\varphi(t) = \sin(10t)$, $\psi(t) = \sin(10t) + 2$ driven by multifractional Brownian motion (mBm) with functional Hurst parameter $H(t) = \frac{1}{5} \sin(2\pi t) + \frac{1}{2}$ (note that the lowest value of the functional Hurst parameter is $H\left(\frac{3}{4}\right) = 0.3$). For more details on mBm, see \cite{ACLV2002} as well as \cite[Lemma 3.1]{DKMR2018} for results on H\"older continuity of its paths. Based on 10000 simulations, the average time for simulating one path is 0.4968714 seconds.
    \begin{figure}
        \centering
        \includegraphics[width = 0.7\textwidth]{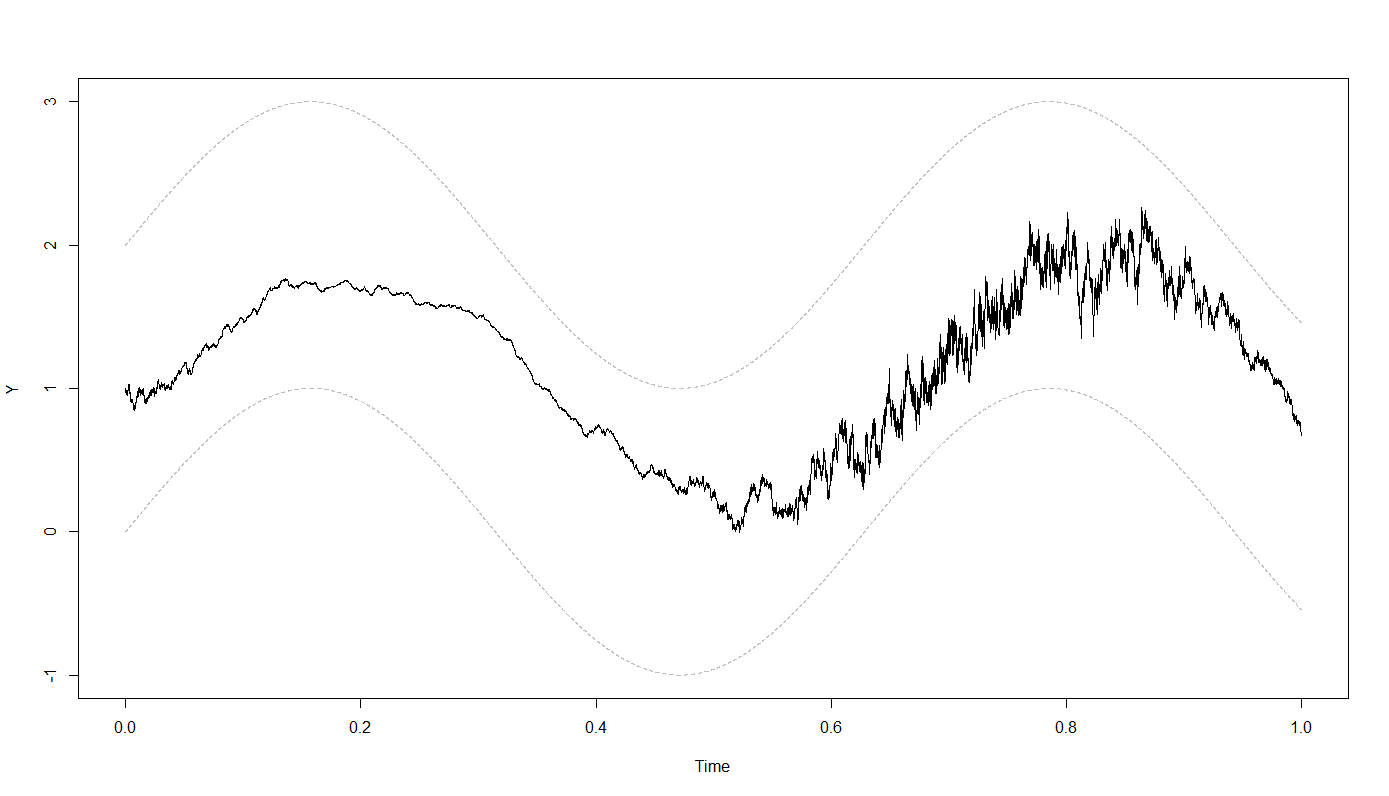}
        \caption{A sample path of \eqref{eq: sandwich mBm} generated using the backward Euler approximation scheme; $N = 10000$, $T=1$, $Y(0) = 1$, $\kappa_1 = \kappa_2 = 1$, $\varphi(t) = \sin(10t)$, $\psi(t) = \sin(10t) + 2$, $Z$ is a multifractional Brownian motion with functional Hurst parameter $H(t) = \frac{1}{5} \sin(2\pi t) + \frac{1}{2}$.}
        \label{fig3}
    \end{figure}
    
%    \begin{figure}
%        \centering
%        \includegraphics[width = \textwidth]{H.png}
%        \caption{Functional Hurst parameter $H(t) = \frac{1}{5} \sin(2\pi t) + \frac{1}{2}$ used in simulation on Fig. \ref{fig4}.}
%        \label{fig4}
%    \end{figure}
\end{example}

\section*{Acknowledgements}

The present research is carried out within the frame and support of the ToppForsk project nr. 274410 of the Research Council of Norway with title STORM: Stochastics for Time-Space Risk Models. The second author was supported by Japan Science and Technology Agency CREST  JPMJCR21. %The second author is supported  by  the Ukrainian research project "Exact formulae, estimates, asymptotic properties and statistical analysis of complex evolutionary systems with many degrees of freedom" (state registration number 0119U100317).

\bibliographystyle{acm}
\bibliography{biblio.bib}

\begin{thebibliography}{10}

\bibitem{Alfi_Coccetti_Petri_Pietronero_2007}
{\sc Alfi, V., Coccetti, F., Petri, A., and Pietronero, L.}
\newblock Roughness and finite size effect in the {NYSE} stock-price
  fluctuations.
\newblock {\em The European physical journal. B 55}, 2 (2007), 135–142.

\bibitem{Alfonsi_2005}
{\sc Alfonsi, A.}
\newblock On the discretization schemes for the {CIR} (and {B}essel squared)
  processes.
\newblock {\em Monte Carlo Methods Appl. 11}, 4 (2005), 355--384.

\bibitem{Alfonsi_2013}
{\sc Alfonsi, A.}
\newblock Strong order one convergence of a drift implicit {E}uler scheme:
  application to the {CIR} process.
\newblock {\em Statist. Probab. Lett. 83}, 2 (2013), 602--607.

\bibitem{Andersen2006}
{\sc Andersen, L. B.~G., and Piterbarg, V.~V.}
\newblock Moment explosions in stochastic volatility models.
\newblock {\em Finance and Stochastics 11}, 1 (Sept. 2006), 29--50.

\bibitem{ACLV2002}
{\sc Ayache, A., Cohen, S., and Vehel, J.~L.}
\newblock The covariance structure of multifractional {B}rownian motion, with
  application to long range dependence.
\newblock In {\em 2000 IEEE International Conference on Acoustics, Speech, and
  Signal Processing. Proceedings (Cat. No.00CH37100)\/} (2002), IEEE.

\bibitem{ASVY2014}
{\sc Azmoodeh, E., Sottinen, T., Viitasaari, L., and Yazigi, A.}
\newblock Necessary and sufficient conditions for {H}ölder continuity of
  {G}aussian processes.
\newblock {\em Statistics \& Probability Letters 94\/} (2014), 230 -- 235.

\bibitem{Bayer_2016}
{\sc Bayer, C., Friz, P., and Gatheral, J.}
\newblock Pricing under rough volatility.
\newblock {\em Quant. Finance 16}, 6 (2016), 887--904.

\bibitem{CKLS}
{\sc Chan, K.~C., Karolyi, G.~A., Longstaff, F.~A., and Sanders, A.~B.}
\newblock An empirical comparison of alternative models of the short-term
  interest rate.
\newblock {\em The journal of finance 47}, 3 (1992), 1209–1227.

\bibitem{Cox1975NotesOO}
{\sc Cox, J.~C.}
\newblock The constant elasticity of variance option pricing model.
\newblock {\em The Journal of Portfolio Management 23}, 5 (1996), 15–17.

\bibitem{COX1981}
{\sc Cox, J.~C., Ingersoll, J.~E., and Ross, S.~A.}
\newblock A re-examination of traditional hypotheses about the term structure
  of interest rates.
\newblock {\em The Journal of Finance 36}, 4 (Sept. 1981), 769--799.

\bibitem{COX1985-1}
{\sc Cox, J.~C., Ingersoll, J.~E., and Ross, S.~A.}
\newblock An intertemporal general equilibrium model of asset prices.
\newblock {\em Econometrica 53}, 2 (Mar. 1985), 363.

\bibitem{COX1985-2}
{\sc Cox, J.~C., Ingersoll, J.~E., and Ross, S.~A.}
\newblock A theory of the term structure of interest rates.
\newblock {\em Econometrica 53}, 2 (Mar. 1985), 385.

\bibitem{Dereich_Neuenkirch_Szpruch_2012}
{\sc Dereich, S., Neuenkirch, A., and Szpruch, L.}
\newblock An {E}uler-type method for the strong approximation of the
  {C}ox–{I}ngersoll–{R}oss process.
\newblock {\em Proceedings of the Royal Society A. Mathematical, physical, and
  engineering sciences 468}, 2140 (2012), 1105–1115.

\bibitem{DNMYT2020}
{\sc Di~Nunno, G., Mishura, Y., and Yurchenko-Tytarenko, A.}
\newblock Sandwiched {SDE}s with unbounded drift driven by {H}ölder noises.
\newblock {\em ArXiv 2012.11465\/} (2020).

\bibitem{Domingo2019}
{\sc Domingo, D., d'Onofrio, A., and Flandoli, F.}
\newblock Properties of bounded stochastic processes employed in biophysics.
\newblock {\em Stochastic Analysis and Applications 38}, 2 (Dec. 2019),
  277--306.

\bibitem{BoundedNoises2013}
{\sc d'Onofrio, A.}, Ed.
\newblock {\em Bounded Noises in Physics, Biology, and Engineering}.
\newblock Springer New York, 2013.

\bibitem{DKMR2018}
{\sc Dozzi, M., Kozachenko, Y., Mishura, Y., and Ralchenko, K.}
\newblock Asymptotic growth of trajectories of multifractional {B}rownian
  motion, with statistical applications to drift parameter estimation.
\newblock {\em Statistical inference for stochastic processes 21}, 1 (2018),
  21–52.

\bibitem{Hong2020}
{\sc Hong, J., Huang, C., Kamrani, M., and Wang, X.}
\newblock Optimal strong convergence rate of a backward {E}uler type scheme for
  the {C}ox–{I}ngersoll–{R}oss model driven by fractional {B}rownian
  motion.
\newblock {\em Stochastic Processes and their Applications 130}, 5 (2020), 2675
  -- 2692.

\bibitem{Hu2008}
{\sc Hu, Y., Nualart, D., and Song, X.}
\newblock A singular stochastic differential equation driven by fractional
  {B}rownian motion.
\newblock {\em Statistics {\&} Probability Letters 78}, 14 (Oct. 2008),
  2075--2085.

\bibitem{Kruse_2014}
{\sc Kruse, R.}
\newblock {\em Strong and weak approximation of semilinear stochastic evolution
  equations}.
\newblock Springer International Publishing, 2014.

\bibitem{Kubilius_Medziunas_2020}
{\sc Kubilius, K., and Medžiūnas, A.}
\newblock Positive solutions of the fractional {SDE}s with non-{L}ipschitz
  diffusion coefficient.
\newblock {\em Mathematics 9}, 1 (2020).

\bibitem{ZhYu2020}
{\sc Zhang, S.-Q., and Yuan, C.}
\newblock Stochastic differential equations driven by fractional {B}rownian
  motion with locally {L}ipschitz drift and their implicit {E}uler
  approximation.
\newblock {\em Proceedings of the Royal Society of Edinburgh: Section A
  Mathematics\/} (2020), 1–27.

\end{thebibliography}

\end{document}